\documentclass[12pt]{article}
\usepackage[T1]{fontenc}
\usepackage{lmodern,amsmath,amsthm,amsfonts,amssymb,graphicx,float,microtype,thmtools,underscore,mathtools,xurl}
\usepackage[dvipsnames,svgnames,table]{xcolor}
\usepackage[shortlabels]{enumitem}
\setlist[itemize]{topsep=0ex,itemsep=0ex,parsep=1ex}
\setlist[enumerate]{topsep=0ex,itemsep=0ex,parsep=1ex}
\usepackage[unicode=true]{hyperref}
\hypersetup{ 
colorlinks,
breaklinks=true,
linkcolor={blue!60!black},
citecolor={black},
urlcolor={blue!60!black},
pdftitle={Optimal Tree-Decompositions with Bags of Bounded Treewidth}}
\usepackage[capitalise, compress, nameinlink, noabbrev]{cleveref}
\crefname{lem}{Lemma}{Lemmas}
\crefname{thm}{Theorem}{Theorems}
\crefname{cor}{Corollary}{Corollaries}
\crefname{open}{Open Problem}{Open Problems}
\newcommand{\defn}[1]{\textcolor{Maroon}{\emph{#1}}}
\newcommand{\mathdefn}[1]{\textcolor{Maroon}{#1}}
\usepackage[longnamesfirst,numbers,sort&compress]{natbib}
\makeatletter
\def\NAT@spacechar{~}
\makeatother
\usepackage[tmargin=30mm,bmargin=30mm,lmargin=30mm,rmargin=30mm]{geometry}
\renewcommand{\baselinestretch}{1.09}
\allowdisplaybreaks
\DeclarePairedDelimiter{\ceil}{\lceil}{\rceil}

\renewcommand{\epsilon}{\varepsilon}
\renewcommand{\emptyset}{\varnothing}

\renewcommand{\le}{\leqslant}
\renewcommand{\geq}{\geqslant}
\renewcommand{\leq}{\leqslant}

\DeclareMathOperator{\tw}{tw}
\DeclareMathOperator{\ttw}{tree-tw}
\DeclareMathOperator{\tpw}{tree-pw}
\DeclareMathOperator{\tbw}{tree-bw}
\DeclareMathOperator{\ttd}{tree-td}
\newcommand{\treef}{\textup{tree-}f}

\DeclareMathOperator{\talpha}{tree-\alpha}
\DeclareMathOperator{\tchi}{tree-\chi}
\DeclareMathOperator{\ltw}{ltw}
\DeclareMathOperator{\pw}{pw}
\DeclareMathOperator{\bw}{bw}
\DeclareMathOperator{\td}{td}

\newcommand{\GG}{\mathcal{G}}

\newcommand{\CC}{\mathcal{C}}
\newcommand{\TT}{\mathcal{T}}
\newcommand{\DD}{\mathcal{D}}

\renewcommand{\thefootnote}{\fnsymbol{footnote}}
\theoremstyle{plain}
\newtheorem{thm}{Theorem}
\newtheorem{lem}[thm]{Lemma}
\newtheorem{cor}[thm]{Corollary}

\newtheorem{prop}[thm]{Proposition}
\newtheorem*{claim}{Claim}
\crefname{obs}{Observation}{Observations}
\newtheorem*{lem*}{Lemma}
\theoremstyle{definition}

\newtheorem*{conj*}{Conjecture}

\begin{document}
\title{\bf\boldmath\fontsize{18pt}{18pt}\selectfont 
Optimal Tree-Decompositions with\\
Bags of Bounded Treewidth}

\author{Kevin Hendrey and David~R.~Wood\,\footnotemark[2]}

\maketitle

\begin{abstract}
We prove that several natural graph classes have tree-decompositions with minimum width such that each bag has bounded treewidth. For example, every planar graph has a tree-decomposition with minimum width such that each bag has treewidth at most 3. This treewidth bound is best possible. More generally, every graph of Euler genus $g$ has a tree-decomposition with minimum width such that each bag has treewidth in $O(g)$. This treewidth bound is best possible. Most generally, every $K_p$-minor-free graph has a tree-decomposition with minimum width such that each bag has treewidth at most some polynomial function $f(p)$. 

In such results, the assumption of an excluded minor is justified, since we show that analogous results do not hold for the class of 1-planar graphs, which is one of the simplest non-minor-closed monotone classes. In fact, we show that 1-planar graphs do not have tree-decompositions with width within an additive constant of optimal, and with bags of bounded treewidth. On the other hand, we show that 1-planar $n$-vertex graphs 
have tree-decompositions with width $O(\sqrt{n})$ (which is the asymptotically tight bound) and with bounded treewidth bags. Moreover, this result holds in the more general setting of bounded layered treewidth, where the union of a bounded number of bags has bounded treewidth.
\end{abstract}

\footnotetext[2]{School of Mathematics, Monash University, Melbourne, Australia (\texttt{Kevin.Hendrey1,david.wood@monash.edu}). Research supported by the Australian Research Council. Research of Wood also supported by NSERC. }

\renewcommand{\thefootnote}{\arabic{footnote}}

\newpage

\section{Introduction}  
\label{Intro}

Tree-decompositions were introduced by \citet{RS-II}, as a key ingredient in their Graph Minor Theory\footnote{We consider simple undirected graphs $G$ with vertex set $V(G)$ and edge set $E(G)$. A graph $H$ is a \defn{minor} of a graph $G$ if a graph isomorphic to $H$ can be obtained from $G$ by a sequence of edge deletions, vertex deletions, and edge contractions. A graph $H$ is a \defn{topological minor} of $G$ if a subdivision of $H$ is a subgraph of $G$. A \defn{graph class} is a collection of graphs closed under isomorphism. A graph class $\GG$ is \defn{minor-closed} if for every graph $G\in\GG$ every minor of $G$ is in $\GG$. A graph class $\GG$ is \defn{monotone} if for every graph $G\in\GG$ every subgraph of $G$ is in $\GG$. A graph class $\GG$ is \defn{proper} if some graph is not in $\GG$.}. Indeed, the dichotomy between minor-closed classes with or without bounded treewidth is a central theme of their work. Tree-decompositions arise in several other results, such as the Erd\H{o}s-P\'osa theorem for planar minors~\cite{RS-V,CvBHJR19}, and  Reed's beautiful theorem on $k$-near bipartite graphs~\citep{Reed99a}. Tree-decompositions are also a key tool in algorithmic graph theory, since many NP-complete problems are solvable in linear time on graphs with bounded treewidth~\citep{Courcelle90}. 

For a non-null tree $T$, a \defn{$T$-decomposition} of a graph $G$ is a collection $(B_x:x \in V(T))$ such that:
\begin{itemize}
    \item $B_x\subseteq V(G)$ for each $x\in V(T)$,
    \item for each edge ${vw \in E(G)}$, there exists a node ${x \in V(T)}$ with ${v,w \in B_x}$, and
    \item for each vertex ${v \in V(G)}$, the set $\{ x \in V(T) : v \in B_x \}$ induces a non-empty (connected) subtree of $T$.
\end{itemize}
The \defn{width} of such a $T$-decomposition is ${\max\{ |B_x| : x \in V(T) \}-1}$. A \defn{tree-decomposition} is a $T$-decomposition for any tree $T$. The \defn{treewidth} of a graph $G$, denoted \defn{$\tw(G)$}, is the minimum width of a tree-decomposition of $G$. Treewidth is the standard measure of how similar a graph is to a tree. Indeed, a connected graph has treewidth at most 1 if and only if it is a tree. See \citep{HW17,Bodlaender98,Reed97} for surveys on treewidth.

In addition to studying the width of tree-decompositions, much recent work has studied tree-decompositions where the subgraph induced by each bag is well-structured in some sense. In this direction, \citet{Adler06} introduced the following definition. For a graph parameter $f$ and graph $G$, let \defn{$\treef(G)$} be the minimum integer $k$ such that $G$ has a tree-decomposition $(B_x:x\in V(T))$ such that $f(G[B_x])\leq k$ for each node $x\in V(T)$. Tree-chromatic number $\tchi$ was introduced by \citet{Seymour16}, and has since attracted substantial interest~\citep{HK17,BFMMSTT19,HRWY21,KRU25}. 
Tree-diameter (under the name `tree-length') was introduced by \citet{DG07}; variants of this parameter have been widely studied in connection with coarse graph theory~\citep{BS24,NSS25,Hickingbotham25,Dragan25,DK25}. 
Tree-independence number $\talpha$ was introduced by \citet{Yolov18} and independently by Dallard, Milani\v{c} and \v{S}torgel in the `Treewidth versus clique-number' series~\citep{DMS21,DMS24a,DMS24b,DKKMMSW24}; it has since been widely studied~\citep{KRU25,LMMORS24,CHMW25,DFGKM24,MR22,KRU25,DMMY25,AMR25,Hickingbotham23}, including the `Tree-independence number' series \citep{CHT24,AACHSV24a,CHLS26,CGHLS25,CC25,CCLMS25}, and the `Induced subgraphs and tree decompositions' series \citep{ACV22,ACDHRSV24,ACHS22,ACHS23,AACHSV24,ACHS25,AACHS24,AACHS24a,ACS25,AACHS25,ACHS25a,ACHS25b,ACHS24,CHS26,CHS24,CHS24a,CGHLS24,CCHS25}.

\citet{LNW} studied $\ttw$ (and to a lesser extent $\tpw$, $\tbw$ and $\ttd$, where $\pw$, $\bw$ and $\td$ denote pathwidth, bandwidth and treedepth respectively). For example, \citet{LNW} showed that for each integer $p$ there is an integer $c$ such that $\ttw(G)\le c$ for every $K_p$-minor-free graph $G$; that is, $G$ has a tree-decomposition such that each bag induces a subgraph with treewidth at most $c$. For the sake of brevity, if $S$ is a set of vertices in a graph $G$, define the \defn{treewidth} of $S$ to be $\tw(G[S])$. 

The primary contribution of this paper is  to show that graphs in any proper minor-closed class admit tree-decompositions that simultaneously have minimum width and every bag has bounded treewidth. A tree-decomposition of a graph $G$ with width $\tw(G)$ is said to be \defn{optimal}. 

First consider planar graphs $G$. \citet{LNW} showed that $G$ has a tree-decompostition with bags of treewidth 3 (that is, $\ttw(G)\leq 3$). The treewidth 3 bound here is best possible whenever $G$ contains $K_4$, since in any tree-decomposition of $G$ each clique of $G$ is contained in a single bag, and $\tw(K_4)=3$. The result of \citet{LNW} gives no bound on the width of this tree-decomposition (and the method used does not relate to the optimal width). We prove the following analogous result with optimal width.

\begin{thm}
\label{PlanarOptimalTW3bags}
Every planar graph has an optimal tree-decomposition in which every bag has treewidth at most 3. 
\end{thm}

\cref{PlanarOptimalTW3bags} is proved in \cref{Planar}, where in fact we give a more precise description of the bags.

Next consider graphs embeddable in a fixed surface\footnote{The \defn{Euler genus} of an orientable surface with $h$ handles is $2h$. The \defn{Euler genus} of  a non-orientable surface with $c$ cross-caps is $c$. The \defn{Euler genus} of a graph $G$ is the minimum Euler genus of a surface in which $G$ embeds (with no crossings).}. \citet{LNW} proved that $\ttw(G)\in O(g)$ for graphs $G$ with Euler genus $g$, and that this $O(g)$ bound is best possible. The next theorem strengthens this result and generalises \cref{PlanarOptimalTW3bags} for graphs embeddable on any fixed surface.

\begin{thm}
\label{EulerGenusOptimalTWgbags}
Every graph with Euler genus $g$ has an optimal tree-decomposition
in which every bag has treewidth at most $\max\{4g+2,3\}$.
\end{thm}

\cref{EulerGenusOptimalTWgbags} is proved in \cref{Surface}, where we give further properties of the tree-decomposition in \cref{EulerGenusOptimalTWgbags}.

More generally, our next contribution says that graphs in any proper minor-closed class have optimal tree-decompositions with bags of bounded treewidth.

\begin{thm}
\label{KtMinorFreeOptimalTWbags}
For any integer $p\geq 1$ there exists an integer $c$ such that every $K_p$-minor-free graph has an optimal tree-decomposition in which every bag has treewidth at most $c$.
\end{thm}

As mentioned above, \citet{LNW} proved the weakening of \cref{KtMinorFreeOptimalTWbags} with no bound on the size of the bags.

The key idea underlying our first two main results (\cref{PlanarOptimalTW3bags,{EulerGenusOptimalTWgbags}}) is surprisingly simple. Consider a tree-decomposition $\DD=(B_x:x\in V(T))$ and a separation $(A_1,A_2)$ of a graph $G$, such that some bag $B_x$ contains $A_1\cap A_2$, and $B_x$ intersects both $A_1\setminus A_2$ and $A_2\setminus A_1$. Let $\DD_1$ and $\DD_2$ be the induced tree-decompositions of $G[A_1]$ and $G[A_2]$ respectively. We can combine $\DD_1$ and $\DD_2$ to obtain a new tree-decomposition of $G$ by adding an edge between the nodes corresponding to $x$ in $\DD_1$ and $\DD_2$. By repeatedly applying this operation, we can obtain a tree-decomposition of $G$ such that there is no separation $(A_1,A_2)$ of $G$ where a bag $B_x$ contains $A_1\cap A_2$, and $B_x$ intersects both $A_1\setminus A_2$ and $A_2\setminus A_1$; we call such a bag `unbreakable'.  For graphs embedded on a surface, every cycle that bounds a disk separates the vertices drawn inside the disk from those drawn outside the disk. Thus, if an unbreakabale bag $B$ contains a cycle bounding a disk $D$, then $B$ must live entirely inside or entirely outside $D$. This allows us to bound the treewidth of $B$. Thus approach also leads to a proof that every graph with maximum degree at most 3 has an optimal tree-decomposition such that every bag has treewidth at most 3 (\cref{Degree3}). 

For $K_p$-minor-free graphs, the situation is complicated by the existence of apex vertices. Indeed, any set that excludes a dominant vertex is unbreakable. We therefore need a more sophisticated method of modifying tree-decompositions in this setting. To this end we introduce the notion of `irreducible' sets. In \cref{sec:BreakabilityReducibility} we explain how to obtain optimal tree-decompositions with irreducible bags, and in \cref{sec:Ktminorfree} we bound the treewidth of subgraphs of $K_p$-minor-free graphs induced on irreducible sets, thus proving \cref{KtMinorFreeOptimalTWbags}.

\subsection{Monotone Graph Classes}

Inspired by the above results for minor-closed classes, it is natural to ask for a given monotone graph class $\GG$, does every graph in $\GG$ have an optimal tree-decomposition such that every bag has bounded treewidth? We show that the answer is `no' for a surprisingly simple graph class. A graph is \defn{1-planar} if it has a drawing in the plane with at most one crossing on each edge; see \citep{KLM17} for a survey on 1-planar graphs. Since every graph has a 1-planar subdivision, 1-planar graphs contain arbitrarily large complete graph minors~\citep{DEW17,HW24,HKW}. On the other hand, we consider 1-planar graphs to be one of the simplest monotone classes that contains subdivisions of any graph. We show that 1-planar graphs do not have optimal tree-decompositions with bags of bounded treewidth. In fact, the same result holds for tree-decompositions of width within an additive constant of optimal.

\begin{thm}
\label{1PlanarIntro}
For any integers $c,w\geq 0$ there is a $1$-planar graph $G$ such that every tree-decomposition of $G$ with width at most $\tw(G)+c$ has a bag with treewidth greater than $w$.
\end{thm}

This result is proved in \cref{1Planar}, where we in fact prove a significant strengthening. 

In contrast to \cref{1PlanarIntro}, we show that 1-planar graphs with $n$ vertices have tree-decompositions with width $O(\sqrt{n})$ such that each bag has bounded treewidth  (\cref{glPlanarSqrtn}). Here the $O(\sqrt{n})$ width bound is best possible, since the $\sqrt{n}\times\sqrt{n}$ grid graph is planar with treewidth $\sqrt{n}$. This shows a marked contrast between the settings of optimal width and $O(\sqrt{n})$ width. Moreover, the upper bound holds in the more general setting of graphs of bounded layered treewidth (\cref{ltw}), where we in fact show that the union of any bounded number of bags has bounded treewidth. This material is presented in \cref{LTW}.

We conclude the paper in  \cref{open} by discussing several open problems that arise from our work. 
The following table summarises all our results. 

\begin{table}[!ht]
\caption{Graphs that have tree-decompositions with bounded 
treewidth bags and with the given width bound. }
\begin{tabular}{l|ccc}
\hline
graph class & optimal & near-optimal & asymptotically tight\\
 & width $\tw(G)$ & width $\tw(G)+c$ & width $O(\sqrt{n})$ \\
\hline
proper minor-closed & yes & yes  & yes\\
1-planar & no & no & yes \\
bounded layered treewidth & no & no & yes \\
\hline
\end{tabular}
\end{table}

\section{Normal, Basic and Refined Tree-Decompositions}

This section introduces some elementary results about tree-decompositions that underpin the main proofs. 

Let $(B_x:x\in V(T))$ be a tree-decomposition of a graph $G$. If $B_x\subseteq B_y$ for some edge $xy\in E(T)$, then let $T'$ be the tree obtained from $T$ by contracting $xy$ into a new vertex $z$, and let $B_z:=B_y$. Then $(B_x:x\in V(T'))$ is a tree-decomposition of $G$ with width equal to the width of $(B_x:x\in V(T))$, and $|V(T')|<|V(T)|$. To \defn{normalise} a tree-decomposition means to apply this operation until $B_x\not\subseteq B_y$ for each edge $xy\in E(T)$. This process terminates, since each contraction decreases $|V(T)|$ by exactly $1$. A tree-decomposition $(B_x:x\in V(T))$ with $B_x\not\subseteq B_y$ for each $xy\in E(T)$ is said to be \defn{normal}. Applying normalisation to an optimal tree-decomposition of a graph $G$ gives a  tree-decomposition of $G$ that is optimal and normal.

Say $(B_x:x\in V(T))$ is a normal tree-decomposition of a graph $G$ with $V(G)\neq\emptyset$. Root $T$ at an arbitrary node $r\in V(T)$ with $B_r\neq\emptyset$. For each edge $xy\in E(T)$ where $y$ is the parent of $x$, let $f(xy)$ be any vertex in $B_x\setminus B_y$. Thus $f$ is an injection from $E(T)$ to $V(G-B_r)$. So $|V(G)|-1\geq |V(G)|-|B_r|\geq|E(T)|=|V(T)|-1$, and $|V(T)|\leq|V(G)|$. That is, every normal tree-decomposition of a graph $G$ has at most $|V(G)|$ bags, which is a well-known fact. 


For a tree-decomposition $\DD=(B_x:x\in V(T))$ of a graph $G$ and a positive integer $i$, let $\mathdefn{n_i(\DD)}$ be the number of nodes $x\in V(T)$ such that $|B_x|=i$, and let $\mathdefn{N(\DD)} :=(n_{|V(G)|}(\DD),n_{|V(G)|-1}(\DD),\dots,n_1(\DD))$.
We say $\DD$ is \defn{basic} if there is no tree-decomposition $\DD'$ of $G$ such that $N(\DD')$
is lexicographically less than $N(\DD)$.

We implicitly use the following simple lemma  throughout the paper.

\begin{lem}\label{lem:basic}
    Every graph has a basic tree-decomposition, and every basic tree-decomposition is normal.
\end{lem}

\begin{proof}
Let $G$ be a graph. Since every normal tree-decomposition of $G$ has at most $|V(G)|$ bags, there are finitely many normal tree-decompositions of $G$. So there exists a normal tree-decomposition $\DD$ of $G$ that lexicographically minimises $N(\DD)$.
    Suppose for contradiction that there is a tree-decomposition $\DD'$ of $G$ such that $N(\DD')$ is lexicographically less than $N(\DD)$.
    Let $\DD''$ be obtained by normalising $\DD'$.
    Then $N(\DD'')$ is lexicographically less than $N(\DD)$, contradicting our choice of $\DD$. 
      This shows that $G$ has a basic tree-decomposition.
      
      Now suppose that $G$ has a basic tree-decomposition $\DD$ that is not normal, and let $\DD'$ be the normalisation of $\DD$.
      Since normalisation only removes bags from $\DD$, $N(\DD')$ is lexicographically less than $N(\DD)$.
\end{proof}

A tree-decomposition $\DD'$ of a graph $G$ is a \defn{refinement} of a tree-decomposition $\DD$ of $G$ if each bag of $\DD'$ is a subset of a bag of $\DD$. If $\DD'$ is a refinement of $\DD$ and not vice versa, then $\DD'$ is a \defn{proper refinement} of $\DD$.
A tree-decomposition is \defn{refined} if it is normal and has no proper refinement.

\begin{lem}\label{lem:refinement}
    Every basic tree-decomposition is refined.
\end{lem}
\begin{proof}
    By \cref{lem:basic}, every basic tree-decomposition is normal.
    Suppose for contradiction that $\DD'$ is a proper refinement of a basic tree-decomposition $\DD$ of a graph $G$.
    Let $\DD''$ be a tree-decomposition of $G$ obtained by normalising $\DD'$, and note that $\DD''$ is also a proper refinement of $\DD$.
    Let $B_x$ be the largest bag of $\DD$ that is not a subset of any bag of $\DD''$.
    Now, for any bag $B_u$ of $\DD$ larger than $B_x$, there is some bag $B''_{v}$ of $\DD''$ and some bag $B_w$ of $\DD$ such that $B_u\subseteq B''_v\subseteq B_w$.
    The second part of \cref{lem:basic} implies  $u=w$, and thus  $B_u=B''_v=B_w$.
    
    Suppose for contradiction that there is a bag $B''_u$ of $\DD''$ of size at least $|B_x|$ that is not a bag of $\DD$.
    Since $\DD''$ is a refinement of $\DD$, there is a bag $B_v$ of $\DD$ containing $B''_u$ as a proper subset.
    By the above argument, $B_v$ is also a bag of $\DD''$, contradicting that $\DD''$ is normal.
    Hence, every bag of $\DD''$ of size at least $|B_x|$ is also a bag of $\DD$.
    Since every bag of $\DD''$ of size at least $|B_x|$ is also a bag of $\DD$ but not vice versa and both are normal tree-decompositions, $N(\DD'')$ is lexicographically less than $N(\DD)$, 
    contradicting that $\DD$ is basic.    
\end{proof}

\cref{lem:basic,lem:refinement} imply that every graph has a refined  tree-decomposition, which we henceforth use implicitly.

For a tree $T$ and edge $xy\in E(T)$, let \defn{$T_{x:y}$} and \defn{$T_{y:x}$} be the subtrees of $T-xy$ respectively containing $x$ and $y$. For a tree-decomposition $(B_x:x\in V(T))$ of a graph $G$ and for each edge $xy\in E(T)$, let $\mathdefn{G_{x:y}}:= G[ \bigcup\{ B_z\setminus B_y :z\in V(T_{x:y}) \}]$. 

\begin{lem}\label{ComponentNormalisation}
    For every refined tree-decomposition $\DD=(B_x:x\in V(T))$ of a graph $G$, for all nodes $x,y\in V(T)$ (not necessarily distinct), for any subset $S\subseteq B_y$, the bag $B_x$ intersects exactly one component of $G-S$, unless $x=y$ and $S=B_y$.
\end{lem}
\begin{proof}
Since $\DD$ is normal, $B_x$ is not a subset of $S$ unless $y=x$ and $S=B_y$. 
Suppose for contradiction that $B_x$ intersects multiple components of $G-S$, and let $H_1$ be one such component.

Let $T_1$ and $T_2$ be disjoint copies of $T$, where $u_i$ is the copy of $u$ in $T_i$, for each $u\in V(T)$ and $i\in\{1,2\}$.
Let $T'$ be obtained from $T_1\cup T_2$ by adding the edge $y_1y_2$.
Let $D_1:=V(H_1)\cup S$ and $D_2:=V(G)\setminus V(H_1)$.
For each node $u_i\in V(T_i)$, let $B_{u_i}:= B_u \cap V(D_i)$. So $\DD':=(B_u:u\in V(T'))$ is a tree-decomposition of $G$, and is a proper refinement of $\DD$ since $B_x$ is not a subset of any bag of $\DD'$.
This contradiction completes the proof.
\end{proof}

\section{Breakability and Reducibility}\label{sec:BreakabilityReducibility}

This section introduces the notions of breakable and reducible sets in a graph, which are key tools in the proofs of our main results. 

A \defn{separation} of a graph $G$ is an ordered pair $(A,B)$ such that $A,B\subseteq V(G)$ and $G=G[A]\cup G[B]$. The \defn{order} of $(A,B)$ is $|A\cap B|$. A separation $(A,B)$ of a graph $G$ \defn{breaks} a set $S\subseteq V(G)$ if $S\setminus A$ and $S\setminus B$ are both nonempty, and $A\cap B\subseteq S$. 
A set $S\subseteq V(G)$ is \defn{breakable} if there exists a separation $(A,B)$ that breaks $S$, otherwise $S$ is \defn{unbreakable}.

\begin{lem}
\label{BasicUnbreakable}
Every bag of every refined tree-decomposition 
$(B_x:x\in V(T))$ of a graph $G$ is unbreakable. 
\end{lem}
\begin{proof}
Consider a separation $(A,B)$ of $G$ and a node $y\in V(T)$ such that $A\cap B\subseteq B_y$ and $B_y\setminus A\neq\emptyset$.
By \cref{ComponentNormalisation}, there is a unique component $C$ of $G-(A\cap B)$ that intersects $B_y$. 
Since $V(C)\cap B\supseteq B_y\setminus A\neq \emptyset$, we have $V(C)\cap A=\emptyset$, so $B_y\setminus B=\emptyset$.
Thus every bag of $(B_x:x\in V(T))$ is unbreakable.
\end{proof}

Together with \cref{lem:refinement}, this yields the following immediate corollary.

\begin{cor}\label{cor:unbreakable}
    Every bag of every basic tree-decomposition is unbreakable.
\end{cor}

A set $S$ of vertices in a graph $G$ is \defn{reducible} if there is a separation $(A,B)$ of $G$ such that both $|(S\cup B)\cap A|$ and $|(S\cup A)\cap B|$ are strictly less than $|S|$, otherwise $S$ is \defn{irreducible}. 

Note that if $S$ is breakable, then it is reducible (since $(A\cap B)\setminus S=\emptyset$).
That is, if $S$ is irreducible, then $S$ is unbreakable. Thus, the following lemma strengthens \cref{cor:unbreakable}.

\begin{lem}
\label{BasicIrreducible}
Every bag of a basic tree-decomposition $(B_x:x\in V(T))$ of a graph $G$ is irreducible.
\end{lem}

\begin{proof}
Suppose for contradiction that for some $y\in V(T)$, there is a separation $(A_1,A_2)$ of $G$ such that both $|(B_y \cup A_1)\cap A_2|$ and $|(B_y\cup A_2)\cap A_1|$ are strictly less than $|B_y|$. Select such a separation of minimum order.
In particular, there is no set $S'$ of size smaller than $|A_1\cap A_2|$ that separates $A_1\cap B_y$ from $A_1\cap A_2$, or a smaller order separation can be obtained by taking $S'$ and the vertices in components of $G-S'$ intersecting $A_1\cap B_y$ on one side and taking $S'$ and all remaining vertices on the other.
Thus, by Menger's Theorem, there is a set $\mathcal{P}_1$ of $|A_1\cap A_2|$ disjoint paths from $A_1\cap A_2$ to $B_y\cap A_1$ in $G$.
By construction each of these paths is internally disjoint from $A_1\cap A_2$, and hence is contained in $G[A_1]$. 
By symmetry, there is a set $\mathcal{P}_2$ of $|A_1\cap A_2|$ disjoint paths from $A_1\cap A_2$ to $B_y\cap A_2$ in $G[A_2]$.

For each $a\in A_1\cap A_2$, let $w_a$ be the closest vertex to $y$ in $T$ such that $a\in B_{w_a}$ (meaning that $w_a=y$ if $a\in B_y$).
Now, for each $w\in V(T)$, define 
\[B'_w:=B_w\cup \{a\in A_1\cap A_2 :w\in V(yTw_a)\}.\]
Here $yTw_a$ is the $yw_a$-path in $T$. 
Note that for every vertex $u\in V(G)$, 
the set $\{w\in V(T):u\in B'_w\}$ is a superset of $\{w\in V(T):u\in B_w\}$ and 
induces a subtree of $T$, so $(B'_w:w\in V(T))$ is a tree-decomposition of $G$.


Let $T''$ be the tree with vertex set $V(T)\times \{1,2\}$ such that vertices $(u,i)$ and $(w,j)$ are adjacent if $i=j$ and $uw\in E(T)$, or $i\neq j$ and $u=w=y$.
For each $(u,i)\in V(T'')$, define $B''_{(u,i)}:=B'_u\cap V(A_i)$.
For each $i\in \{1,2\}$, let $T_i:=T[V(T)\times \{i\}]$ and note that $(B''_x:x\in V(T_i))$ is a tree-decomposition of $G[A_i]$.
Since $A_1\cap A_2\subseteq B''_{(v,1)}\cap B''_{(v,2)}$, it follows that $(B''_x:x\in V(T''))$ is a tree-decomposition of $G$.

\begin{claim}
    For each $(v,i)\in V(T'')$, we have $|B''_{(v,i)}|\leq |B_v|$. Furthermore, if equality holds, then $|B''_{(v,3-i)}|< |B_y|$.
\end{claim}
\begin{proof}
    Let $X_v:=B'_v\setminus B_v$, and let $j:=3-i$.
    By construction, $X_v\subseteq (A_1\cap A_2)\setminus B_y$.
    For each $a\in X'_v$, consider the path $P_{a,j}$ in $\mathcal{P}_{j}$ that has $a$ as an endpoint.
    Since both $B_{w_a}$ and $y$ contain a vertex of $P_{a,j}$ and $v\in V(yTw_a)$, it follows that $B_v$ contains a vertex of $P_{a,j}$. In particular, since $a\notin B_v$, we see that $B_v$ contains a vertex in $V(P_{a,j})\setminus (A_1\cap A_2)$.
    Thus $B_v\setminus A_i$ contains at least $|X_v|$ vertices, and so $|B''_{(v,i)}|= |B_v\cap A_i|+|X_v|\leq |B_v|$.
    If equality holds, then $|B_v\setminus A_i|=|X_v|\leq |(A_1\cap A_2)\setminus B_y|$.
    This means 
    \[|B''_{(v,j)}|\leq |B_v\setminus A_i|+|A_1\cap A_2|\leq |(A_1\cap A_2)\setminus B_y| +|A_1\cap A_2|.\] 
    Recall that $|(B_y \cup A_1)\cap A_2|<|B_y|$, meaning $|(A_1\setminus A_2)\setminus B_y|<|B_y\setminus A_2|$.
    Thus
    \[|B''_{(v,j)}|<|B_y\setminus A_2|+|A_1\cap A_2|=|(B_y\cup A_2)\cap A_1|<|B_y|,\]
    as required.  
\end{proof}
Note that $|B''_{(y,1)}|=|(A_1\cap A_2)\cup (B_y\setminus A_2)|<|B_y|$, and similarly $|B''_{(y,2)}|<|B_y|$.
Let $k$ be the maximum integer such that for some $v\in V(T)$ with $|B_v|=k$ we have $|B''_{v,1}|<k$ and $|B''_{(v,2)}|<k$ (so $k\geq |B_y|$).

By the above claim, the number of bags of $\DD''$ of size $k'$ equals the number of bags of $\DD$ of size $k'$ for all $k'>k$, and $\DD''$ has strictly fewer bags of size $k$ than $\DD$. This contradicts the fact that $\DD$ is basic.
\end{proof}

As an aside, breakability leads to the following results about graphs of bounded maximum degree.

\begin{prop}
\label{Degree}
Every graph $G$ with maximum degree $\Delta$ has an optimal tree-decomposition $\DD$ such that for every bag $B$ of $\DD$, $G[B]$ has maximum degree at most $\Delta-1$ or $|B|\leq\Delta+1$.
\end{prop}

\begin{proof}
Let $\DD$ be a refined tree-decomposition of $G$ with width $\tw(G)$. Consider a bag $B$ of $\DD$. Suppose for the sake of contradiction that $|B|\geq\Delta+2$ and there is a vertex $v\in B$ with degree $\Delta$ in $G[B]$. Thus $N_G[v]\subseteq B$ and some vertex in $B$ is not adjacent to $v$. Hence $(N_G[v], V(G)\setminus \{v\})$ is a separation of $G$ that breaks $B$, which contradicts \cref{BasicUnbreakable}. 
\end{proof}

Since every graph with at most four vertices or with maximum degree at most 2 has treewidth at most 3, \cref{Degree} implies:

\begin{cor}
\label{Degree3}
Every graph $G$ with maximum degree at most 3 has an optimal tree-decomposition $\DD$ such that every bag of $\DD$ has treewidth at most 3.
\end{cor}

\section{Planar Graphs}
\label{Planar}

This section proves \cref{PlanarOptimalTW3bags} showing that every planar graph has an optimal tree-decomposition with bags of treewidth 3. 

We employ the following definitions introduced by \citet{DF21}. A planar graph $G$ is \defn{separable} with respect to a plane embedding $\Pi$ of $G$ if there is a cycle $C$ in $G$ such that there is a vertex of $G-V(C)$ in the interior of $C$ with respect to $\Pi$, and there is a vertex of $G-V(C)$ in the exterior of $C$ with respect to $\Pi$. Otherwise, $G$ is \defn{non-separable} with respect to $\Pi$. That is, for any cycle $C$ in $G$, all the vertices of $G-V(C)$ are in the interior of $C$, or all the vertices of $G-V(C)$ are in the exterior of $C$. 

Let $\Pi$ be a plane embedding of a planar graph $G$. Suppose that $S\subseteq V(G)$ and $G[S]$ is separable with respect to the plane embedding of $G[S]$ induced by $\Pi$. So there is a cycle $C$ in $G[S]$ with at least one vertex of $S$ in the interior of $C$, and at least one vertex of $S$ in the exterior of $C$. Let $A$ be the set of all vertices of $G$ in $C$ or in the interior of $C$. Let $B$ be the set of all vertices of $G$ in $C$ or in the exterior of $C$. By the Jordan Curve Theorem, $(A,B)$ is a separation of $G$ that breaks $S$. So if a set $S\subseteq V(G)$ is unbreakable, then $S$ is non-separable with respect to $\Pi$.  \cref{BasicUnbreakable} thus implies:

\begin{lem}
\label{PlanarNonSep}
Let $\Pi$ be a plane embedding of a planar graph $G$. For every refined tree-decomposition $(B_y:y\in V(T))$ of $G$, for each $y\in V(T)$, $G[B_y]$ is non-separable with respect to the embedding of $G[B_y]$ induced by $\Pi$.
\end{lem}

A planar graph $G$ is \defn{non-separable} if $G$ is non-separable with respect to some plane embedding of $G$. \citet{DF21} showed that the class of non-separable planar graphs is minor-closed, and that a graph $G$ is non-separable planar if and only if $G$ does not contain $K_1 \cup K_4$ or $K_1 \cup K_{2,3}$ or $K_{1,1,3}$ as a minor. In fact, \citet{DF21} provided the following precise structural characterisation: any non-separable planar graph is either outerplanar, or a subgraph of a wheel, or a subgraph of an \defn{elongated triangular prism} (a graph obtained from the triangular prism $K_3\square K_2$ by subdividing edges that are not in triangles any number of times). 

\begin{figure}[!ht]
(a) \hspace*{-4mm} \includegraphics{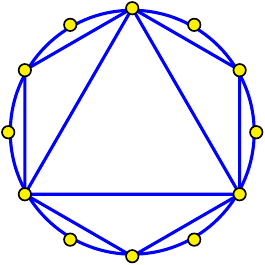}
\quad 
(b) \hspace*{-4mm}\includegraphics{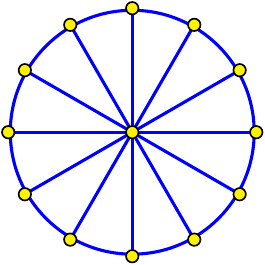}
\quad 
(c) \hspace*{1mm}\includegraphics{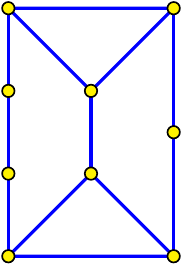}
\caption{Non-separable planar graphs: (a) outerplanar, (b) wheel, (c) elongated triangular prism.}
\end{figure}

\begin{lem}\label{PlanarDetailed}
For every refined tree-decomposition $(B_y:y\in V(T))$ of a planar graph $G$, for each $y\in V(T)$, $G[B_y]$ is outerplanar, is a subgraph of a wheel, or is a subgraph of an elongated triangular prism.
\end{lem}


Each graph listed in \cref{PlanarDetailed} has treewidth at most 3. 
So \cref{PlanarDetailed} implies the following result, which implies \cref{PlanarOptimalTW3bags}.

\begin{cor}
\label{PlanarGraphsCor}
Every planar graph $G$ has an optimal tree-decomposition $(B_x:x\in V(T))$ such that for each $x\in V(T)$, the subgraph $G[B_x]$ has treewidth at most 3; in particular, $G[B_x]$ is outerplanar, a subgraph of a wheel, or a subgraph of an elongated triangular prism. 
\end{cor}

A forthcoming companion paper~\citep{HHHW} gives a more precise structural characterisation than \cref{PlanarGraphsCor}, where  we drop the requirement of outerplanarity.

\section{Graphs on Surfaces}
\label{Surface}

This section proves \cref{EulerGenusOptimalTWgbags} showing that  every graph $G$ with Euler genus $g$ has an optimal tree-decomposition with bags of treewdith $O(g)$. For a graph $G$ embedded in a surface $\Sigma$, for distinct vertices $a,b\in V(G)$, two $ab$-paths $P_1$ and $P_2$ in $G$ are \defn{homotopic} if $P_1 \cup P_2$ bounds a disc in $\Sigma$. We use the following lemma of \citet{MM92} (see Proposition~4.2.7~in~\citep{MoharThom} and the discussion that follows). 

\begin{lem}[\citep{MoharThom,MM92}]
\label{HomotopicPair}
For any graph $G$ embedded in a surface with Euler genus $g\geq 1$, for any distinct vertices $a,b\in V(G)$, if $\mathcal{P}$ is a set of pairwise internally disjoint $ab$-paths in $G$ and
$|\mathcal{P}| \geq 2g+1$, then $\mathcal{P}$ contains a pair of homotopic paths. 
\end{lem}

\begin{lem}
\label{HomotopicTuple}
For any integer $c\geq 2$, for any graph $G$ embedded in a surface with Euler genus $g\geq 1$, for any distinct vertices $a,b\in V(G)$, if $\mathcal{P}$ is a set of pairwise internally disjoint $ab$-paths in $G$ and
$|\mathcal{P}| \geq (c-1)(2g+1)$, then $\mathcal{P}$ contains a set of $c$ pairwise homotopic paths.
\end{lem}

\begin{proof}
Homotopy between paths in $\mathcal{P}$ is an equivalence relation. If some equivalence class has $c$ paths then we are done. Assume each equivalence class at most $c-1$ paths. 
The number of equivalence classes is at least $2g+1$. Taking one path from each equivalence class gives $2g+1$ pairwise non-homotopic $ab$-paths, contradicting \cref{HomotopicPair}.
\end{proof}

\begin{lem}
\label{FindSepCycle}
Every embedding of $K_{2,4g+2}$ in a surface of Euler genus $g\geq 1$ has a cycle bounding a disc and separating two vertices. 
\end{lem}

\begin{proof}
Let $n:=4g+2$. Let $a,b$ be the two vertices of degree $n$ in $K_{2,n}$. There are $n$ internally disjoint $ab$-paths $P_1,\dots,P_n$ in $K_{2,n}$. By \cref{HomotopicTuple}, three of $P_1,\dots,P_n$ are pairwise homotopic. Without loss of generality, $P_1,P_2,P_3$ are pairwise homotopic, and  there is a disc $D$ in $\Sigma$ bounded by $P_1\cup P_2$ with the internal vertex of $P_3$ in the interior of $D$. Choose $D$ to be minimal with this property. Then no other vertex is in the interior of $D$. Since $n\geq 4$, $P_1\cup P_2$ is separating. 
\end{proof}

\begin{lem}
\label{SurfaceBreakable}
For every graph $G$ of Euler genus $g\geq 1$, for any $S\subseteq V(G)$, if $K_{2,4g+2}$ is a minor of $G[S]$, then $S$ is breakable. 
\end{lem}

\begin{proof}
Fix an embedding of $G$ in a surface $\Sigma$ of Euler genus $g$. 
Let $H$ be obtained from $G[S]$ by contracting each branch set of the $K_{2,4g+2}$ minor. 
We obtain an embedding of $H$ in $\Sigma$.
By \cref{FindSepCycle}, there is a cycle $C_0$ in $H$ bounding a disc and separating two vertices in $H$. The branch sets corresponding to $C_0$ thus contain a cycle $C$ in $G[B_x]$ bounding a disc $D$ and separating two vertices in $S$. 
Let $A$ be the set of all vertices of $G$ embedded in $D$ (including $C$).
Let $B$ be the set of all vertices of $G$ embedded in $\Sigma-D$ plus $V(C)$.
So $(A,B)$ is a separation of $G$ that breaks $S$. 
\end{proof}

We use the following lemma by \citet{LY25} (improving on a previous bound by \citet{LeafSeymour15}). A graph $H$ is an \defn{apex-forest} if $H-v$ is a forest for some $v\in V(H)$. 


\begin{lem}[\citep{LY25}]
\label{ApexForest}
For any apex-forest $H$ with $|V(H)|\geq 2$, every $H$-minor-free graph has treewidth at most $|V(H)|-2$. 
\end{lem}


\begin{lem}
\label{SurfaceMain}
For any integer $g\geq 1$, every graph $G$ with Euler genus at most $g$, every refined tree-decomposition $(B_x:x\in V(T))$ of $G$ and every $x\in V(T)$:
\begin{itemize}
\item $G[B_x]$ has no $K_{2,4g+2}$ minor, and
\item $\tw(G[B_x])\leq 4g+2$.
\end{itemize}
\end{lem}

\begin{proof}
    \cref{BasicUnbreakable,SurfaceBreakable} imply that $G[B_x]$ is $K_{2,4g+2}$-minor-free for each $x\in V(T)$. Since  $K_{2,4g+2}$ is an apex-forest, by \cref{ApexForest}, every $K_{2,4g+2}$-minor-free graph has treewidth at most $4g+2$. The result follows.
\end{proof}

\cref{SurfaceMain} implies \cref{EulerGenusOptimalTWgbags} (using \cref{PlanarOptimalTW3bags} in the $g=0$ case). Moreover, the $O(g)$ bound on $\tw(G[B_x])$ in \cref{EulerGenusOptimalTWgbags} is best possible, as we now explain. \citet[Theorem~4.37]{LNW} showed that for sufficiently large $n$, there is a $256$-regular $n$-vertex graph $G$ with $\ttw(G)\geq \frac{n}{4}$. Let $g$ be the Euler genus of $G$. So $g\leq 2 |E(G)|\leq 256 n$. Hence $\ttw(G)\geq \frac{n}{4} \geq \frac{g}{1024}$. In short, there are graphs $G$ with Euler genus $g$ and $\ttw(G)\in\Omega(g)$. That is, in any tree-decomposition of $G$ (regardless of the size of the bags), some bag induces a subgraph with treewidth $\Omega(g)$. Nevertheless,  grid minors in bags can be limited using the following lemma by \citet{GRS04}.

\begin{lem}[{\protect\citep[Lemma~4]{GRS04}}]
\label{GRS}
Let $t,k,n$ be positive integers such that $n \geq t(k+1)$. 
Let $G$ be an $n \times n$ grid graph. If $G$ is embedded in a surface $\Sigma$ of Euler genus at most $t^2-1$, then some $k \times k$ subgrid of $G$ is embedded in a closed disc $D$ in $\Sigma$ such that the boundary cycle of the $k \times k$
grid is the boundary of $D$.
\end{lem}

\begin{lem}
\label{SurfaceGridBreakabale}
For any graph $G$ of Euler genus at most $g$, if $S\subseteq V(G)$ and the $4\ceil{\sqrt{g+1}} \times 4\ceil{\sqrt{g+1}}$ grid is a minor of $G[S]$, then $S$ is breakable.
\end{lem}

\begin{proof}
Let $t:=\ceil{\sqrt{g+1}}$.
So $g\leq t^2-1$, and $G$ embeds in a surface $\Sigma$ of Euler genus at most $t^2-1$. Let $k:=3$ and $n:=4t$. 
By assumption, the $n\times n$ grid is a minor of $G[S]$. 
Let $G'$ be the $n\times n$ grid, embedded in $\Sigma$, obtained from $G[S]$ by appropriate contractions and deletions. 
By \cref{GRS}, some  $3 \times 3$ subgrid of $G'$ is embedded in a closed disc $D_0$ in $\Sigma$ such that the boundary cycle $C_0$ of the subgrid is the boundary of $D_0$. Take $G'$ minimal. Since $n\geq 4$, some vertex of $G'$ is embedded in $\Sigma-D_0$. The branch sets corresponding to $C_0$ contain a cycle $C$ in $G[S]$ bounding a disc $D$, where there is at least one vertex of $S$ in the interior of $D$ (corresponding to the internal vertex of the $3\times 3$ grid), and there is at least one vertex of $S$ in $\Sigma-D$. Let $A$ be the set of all vertices of $G$ embedded in $D$ (including the boundary). 
Let $B$ be the set of all vertices of $G$ embedded in $\Sigma-D$ plus the vertices on the boundary of $D$. So $(A,B)$ is a separation of $G$ that breaks $S$.
\end{proof}

\cref{SurfaceGridBreakabale} implies that in addition to the properties in \cref{SurfaceMain}, in every refined tree-decomposition of a graph with Euler genus $g$, the subgraph induced by each bag has no 
$4\ceil{\sqrt{g+1}} \times 4\ceil{\sqrt{g+1}}$ grid minor. In particular:

\begin{thm}
For any integer $g\geq 1$, every graph $G$ with Euler genus at most $g$ has  an optimal tree-decomposition $(B_x:x\in V(T))$, such that for each $x\in V(T)$:
\begin{itemize}
\item $G[B_x]$ has no $K_{2,4g+2}$ minor,
\item $\tw(G[B_x])\leq 4g+2$,
\item $G[B_x]$ has no 
$4\ceil{\sqrt{g+1}} \times 4\ceil{\sqrt{g+1}}$ grid minor. 
\end{itemize}
\end{thm}

\section{\boldmath \texorpdfstring{$K_p$}{K_p}-Minor-Free Graphs}
\label{sec:Ktminorfree}

This section proves \cref{KtMinorFreeOptimalTWbags}, which shows that for fixed $p$, every $K_p$-minor-free graph has an optimal tree-decomposition with bags of bounded treewidth. 

We begin by defining walls. 
For positive integers $m$ and $n$, the \defn{$m\times n$-grid} is the graph with vertex set $\{1,\dots,m\}\times \{1,\dots,n\}$ such that vertices $(i,j)$ and $(i',j')$ are adjacent if and only if $|i-i'|+|j-j'|=1$. The \defn{elementary $h$-wall} $W_h$ is the graph obtained from the $2h\times h$-grid by deleting every edge of the form $(i,j)(i+1,j)$ with $i+j$ even, and then deleting the two vertices of degree $1$ in the resulting graph.
For each $i\in [k]$, the path of $W_h$ induced by the vertices of the form $(2i-1,j)$ together with the vertices of the form $(2i,j)$ is a \defn{column} of $W_h$, and the path induced by the vertices of the form $(j,i)$ is a row of $W_h$. An \defn{$h$-wall} is a graph that is isomorphic to a subdivision of the elementary $h$-wall. We always implicitly fix such an isomorphism. We use the following variant of the Grid Minor Theorem of \citet{RS-V}.

\begin{lem}[\citep{RS-V}]
\label{lem:WallSubgraph}
There is a function $f$ such that for any integer $k\geq 1$, every graph of treewidth at least $f(k)$ contains a $k$-wall as a subgraph.
\end{lem}

We use the following precursor to the Flat Wall Theorem of \citet{RS-XIII}, which depends on the following definitions. The \defn{rows} and \defn{columns} of an $h$-wall are the paths corresponding to the rows and columns of $W_h$, respectively.
Any graph that is a $k$-wall for some $k$ is a \defn{wall}, and we say its \defn{height} is $k$. A \defn{subwall} of $W_h$ is a subgraph $H\subseteq W_h$ such that $H$ is a wall, each column of $H$ is contained in a column of $W_h$ and each row of $H$ is contained in a row of $W_h$.
The subwalls of a wall $W$ of height $h$ are the corresponding subgraphs of $W$.
The \defn{perimeter} of $W_h$ is the unique shortest cycle containing all degree $2$ vertices, and the \defn{perimeter} of an $h$-wall $W$ is the corresponding cycle of $W$.
Given a wall $W$ in a graph $G$, a subwall $H$ of $W$ is \defn{dividing} if every path in $G$ from $H$ to the perimeter of $W$ contains a vertex in the perimeter of $H$.  
For a wall $W$, define a distance function \defn{$\hat{d}_W$}, where  $\hat{d}_W(v,v)=0$ for all $v\in V(W)$, and for distinct $v,w\in V(W)$, $\hat{d}_W(v,w)$ is the minimum integer $k$ such that for some planar embedding of $W$ there is curve in the plane from $v$ to $w$ that  intersects only $k$ points of the drawing (including $v$ and $w$).  

\begin{thm}[{\protect\citep[(9.3)]{RS-XIII}}]
\label{thm:preflatwall}
For any integer $p\geq 1$ there exists $k,r\geq 0$ such that given a wall $W$ in a $K_p$-minor-free graph $G$ and subwalls $H_1,H_2, \dots , H_t$ with $\hat{d}_W(H_i,H_j)\geq r$ for all distinct $i,j\in \{1,\dots ,t\}$, there is a subset $X\subseteq V(G)$ of size at most $\binom{p}{2}$ and a subset $I\subseteq \{1,\dots ,t\}$ of size at least $t-k$ such that for all $i\in I$, $H_i$ is dividing in $(G-X)\cup W$.
\end{thm}

\cref{BasicIrreducible} and the next lemma imply \cref{KtMinorFreeOptimalTWbags}.

\begin{lem}
\label{IrreducibleBoundedTreewidth}
For any integer $p\geq 1$ there exists an integer $c_p$ such that for every  $K_p$-minor-free graph $G$, every irreducible set $S\subseteq V(G)$ satisfies $\tw(G[S])\leq c_p$.
\end{lem}

\begin{proof}
    Let $r,k\geq 0$ be integers as in \cref{thm:preflatwall}, let $p^*=\max\{\binom{p}{2},2r(k+1)+(k+1)(p+2)\}$, and let $c_p=f(p^*)$, where $f$ is from \cref{lem:WallSubgraph}. 
    Assume that $\tw(G[S])\geq c_p$.
    Our goal is to show that $S$ is reducible. 
    By \cref{lem:WallSubgraph}, $G[S]$ contains a $p^*$-wall $W$ as a subgraph.  
    As illustrated in \cref{Wall}, in $W$ there are:
    \begin{itemize}
        \item $k+1$ pairwise disjoint subwalls $H_1,\dots, H_{k+1}$, each  of height $p+2$ and disjoint from the perimeter of $W$, and
        \item for each $i\in\{1,\dots,k+1\}$ there is a set $\mathcal{C}_i$ of $r$ pairwise disjoint cycles in $W$,  each disjoint from $H_1\cup\dots\cup H_{k+1}$, such that each $C\in\mathcal{C}_i$ separates $H_i$ and $\bigcup\{H_j:j\neq i]\}$.        
    \end{itemize}
    \begin{figure}[!ht]
    \includegraphics[width=\textwidth]{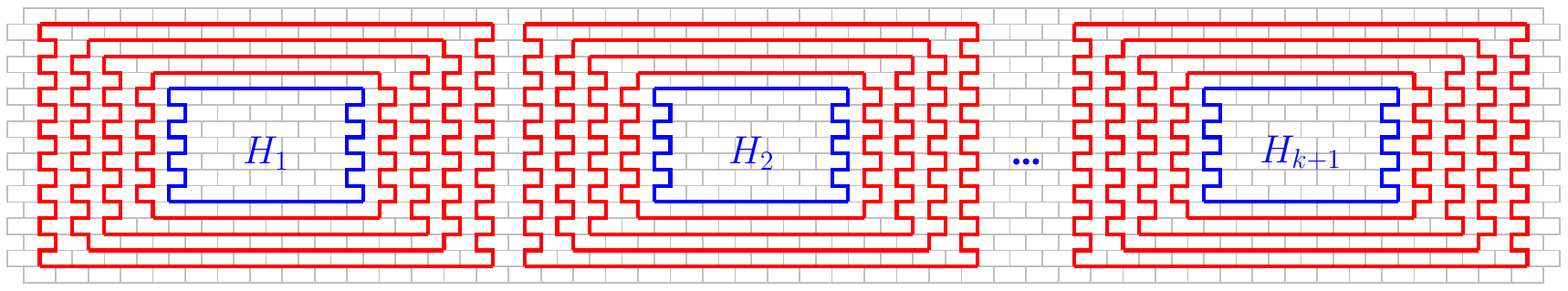}
    \caption{Subwalls $H_1,\dots,H_{k+1}$, and cycle collections $\mathcal{C}_1,\dots,\mathcal{C}_{k+1}$.}
    \label{Wall}
    \end{figure}

    Since $W$ is a subdivision of a 3-connected planar graph, by a theorem of \citet{Whitney-AJM33c}, $W$ has a unique embedding in the plane (up to the choice of the outerface). In this embedding, any curve in the plane from $H_i$ to $H_j$ (for $i\neq j$) must intersect each cycle $C\in \mathcal{C}_i$. Thus $\hat{d}_W(H_i,H_j)\geq r$, and so by \cref{thm:preflatwall} there is some $X\subseteq V(G)$ of size at most $\binom{p}{2}$ and some $i\in \{1,\dots, k+1\}$ such that $H_i$ is dividing in $(G-X)\cup W$.
    Let $S'$ be the union of $X$ and the vertex set of the perimeter of $H_i$. 
    Let $A$ be the union of $S'$ and the vertex sets of all components of $G-S'$ that intersect the perimeter of $W$. Let $B:= V(G)\setminus (A\setminus S')$. 
    So $(A,B)$ is a separation of $G$ with $A\cap B=S'$ and $V(H_i)\subseteq B$.
    Let $X_W$ be the vertex set of the perimeter of $W$, and note that $|(S\cup A)\cap B|\leq |S|-|X_W|+|X|$, which is less than $|S|$ since $|X_W|>p^*\geq \binom{p}{2}\geq |X|$. 
    Likewise $|(S\cup B)\cap A|\leq |S|-|V(H_i)|+|S'|$, which is less than $|S|$ since there are at least $p^2$ vertices of $H_i$ that are not in the perimeter of $H_i$, there are at most $\binom{p}{2}$ vertices of $S'$ that are not in the perimeter of $H_i$, and all vertices in the perimeter of $H_i$ are in $S\cap S'$.
    Thus $S$ is reducible, as required.
\end{proof}

Much work has gone into optimising the bounds in the Grid Minor Theorem~\citep{CC16,CT21,Chuzhoy15} and the Flat Wall Theorem~\citep{KTW18,GSW25,Chuzhoy15}. These results imply polynomial bounds for the functions in \cref{lem:WallSubgraph} and \cref{thm:preflatwall}, and thus for the implicit function in \cref{IrreducibleBoundedTreewidth}. See \citep{GSW25} for an in-depth discussion of these results.

\section{1-Planar Graphs}
\label{1Planar}

This section proves our negative results for 1-planar graphs introduced in \cref{Intro}. We need the following lemma. 

\begin{lem}\label{lem:verywelllinked}
    Let $c$ be a positive integer, $G$ be a graph and $S$ be a set of vertices in $G$ such that for any collection $\{\{a_1,b_1\},\dots, \{a_z,b_z\}\}$ of disjoint pairs of vertices, there is a collection $\{P_{i,j}:i\in \{1,\dots, z\},j\in \{1,\dots, 2c+4\}\}$ of internally disjoint paths, each internally disjoint from $S$, such that for each $i\in \{1,\dots, z\}$ and $j\in \{1,\dots, 2c+4\}$ the endvertices of $P_{i,j}$ are $a_i$ and $b_i$.
    For any tree-decomposition of $G$ of width less than $|S|+c$, some bag contains $S$.
\end{lem}

\begin{proof}
    Consider any tree-decomposition $\DD=(B_x:x\in V(T))$ of $G$ in which no bag contains $S$. 
Our goal is to show that $\DD$ has width at least $|S|+c$.
For every edge $xy\in E(T)$, orient $xy$ towards $x$ if $G_{x:y}$ has more vertices of $S$ than $G_{y:x}$, orient $xy$ towards $y$ if $G_{y:x}$ has more vertices of $S$ than $G_{x:y}$, and otherwise orient $xy$ arbitrarily.
Let $t_0$ be a sink of the resulting oriented tree, and let $t_1,\dots ,t_d$ be the neighbours of $t_0$.
For each $i\in \{1,\dots , d\}$, let $S_i:=S\cap V(G_{t_i:t_0})$, and let $S_0:=S\setminus \bigcup \{S_1,\dots,S_d\}$.
Without loss of generality, $|S_1|=\max\{|S_1|,\dots ,|S_d|\}$.

First consider the case where $2|S_1|\geq |S\setminus S_0|$.
Since $t_0$ is a sink, $|S_1|\leq |S\setminus B_{t_1}|$.
Thus, there are disjoint pairs $\{a_1,b_1\},\dots ,\{a_{|S_1|},b_{|S_1|}\}$, where each $a_i$ is in $S_1$ and each $b_i$ is in $S\setminus B_{t_1}$.
By assumption, there is a set of internally disjoint paths $\{P_{i,j}:i\in \{1,\dots,|S_1|\},j\in \{1,\dots, 2c+4\}\}$, each internally disjoint from $S$, such that for each $i\in  \{1,\dots,|S_1|\}$ and $j\in \{1,\dots, 2c+4\}$ the endpoints of $P_{i,j}$ are $a_i$ and $b_j$.
By the definition of a tree-decomposition, $B_{t_0}\cap B_{t_1}$ separates $S_1$ from $S\setminus B_{t_1}$ in $G$, and so $B_{t_0}\cap B_{t_1}$ contains a vertex of each of these paths.
Since $\{a_1,\dots,a_{|S_1|}\}\cap B_{t_0}=\emptyset$ and $\{b_1,\dots,b_{|S_1|}\}\cap B_{t_1}=\emptyset$, we have $|(B_{t_0}\cap B_{t_1})\setminus S|\geq (2c+4)|S_1|$.
Thus $|B_{t_0}|\geq |S_0|+(2c+4)|S_1|\geq |S_0|+(c+2)|S\setminus S_0|$. By assumption $S$ is not a subset of $B_{t_0}$, so $|S\setminus S_0|\geq 1$ and $|B_{t_0}|\geq |S|+c+1$.

Now consider the case where $2|S_1|<|S\setminus S_0|$.
Let $\mathcal{X}=\{\{a_1,b_1\},\dots ,\{a_z,b_z\}\}$ be maximum sized collection of disjoint pairs of vertices such that each set $\{a_i,b_i\}$ intersects exactly two sets in $\{A_1,\dots,A_d\}$.
Suppose for contradiction that ${2z\leq |\bigcup \{A_1,\dots ,A_d\}|-2}$.
Then for some $i\in \{1,\dots ,d\}$ there are at least two vertices $v$ and $w$ that are not in a pair in $\mathcal{X}$, and all vertices in $S\setminus (S_0\cup S_i)$ are in a pair in $\mathcal{X}$.
Since $|S_i|\leq S_1$ and $2|S_1|<|S\setminus S_0|$, we have $|S_i|< |S\setminus (S_0\cup S_i)|$, there is a pair $\{a_j,b_j\}$ in $\mathcal{X}$ that is disjoint from $S_i$. 
Replacing $\{a_j,b_j\}$ by $\{v,a_j\}$ and $\{b_j,w\}$ increases the size of $\mathcal{X}$, yielding the desired contradiction.
Thus, $2z\geq | A_1\cup\dots\cup A_d|-1=|S\setminus S_0|-1$.
By construction, there is a set of internally disjoint paths $\{P_{i,j}:i\in \{1,\dots,z\},j\in \{1,\dots, 2c+4\}\}$, each internally disjoint from $B_{t_0}$, such that for each $i\in  \{1,\dots,|S_1|\}$ and $j\in \{1,\dots, 2c+4\}$ the endpoints of $P_{i,j}$ are $a_i$ and $b_j$.
Note that for distinct $i,j\in \{1,\dots ,d\}$, the set $B_{t_0}$ separates $A_i$ from $A_j$ in $G$.
Thus, $B_{t_0}$ contains an internal vertex of each of these paths, and so $|B_{t_0}|\geq |S_0|+(2c+4)z\geq |S_0|+(c+2)(|S\setminus S_0|-1)$.
Since $2|S_1|<|S\setminus S_0|$, we have $|S\setminus S_0|\geq 3$.
Thus $|B_{t_0}|\geq (|S|-1)+2(c+1)> |S|+c+1$.

In both cases, the width of $\DD$ is at least $|S|+c \geq \tw(G)+c$. Thus, every tree-decomposition of $G$ with width less than $\tw(G)+c$ has a bag containing $S$.
\end{proof}

The following result implies and strengthens \cref{1PlanarIntro} by taking $G_0$ to be any graph with sufficiently large treewidth (such as a large complete graph or large grid). 

\begin{thm}
For any integer $c\geq 0$ and graph $G_0$ there is a $1$-planar graph $G$ such that every tree-decomposition of $G$ with width less than $\tw(G)+c$ has a bag $B$ such that $G_0$ is a topological minor of $G[B]$.
\end{thm}

\begin{proof}
Draw $G_0$ in the plane (allowing crossings), and let $H$ be the planarisation of this drawing (that is, introduce a new vertex in $H$ at each crossing). 
Let $T^*$ be a rooted spanning tree of the dual of $H$. 
Let $f_0,f_1,\dots, f_s$ be the faces of $H$ (equivalently, vertices of $T^*$) ordered so that $f_s$ is the root of $T^*$, and for each edge $f_if_j$ in $T^*$, if $i<j$ then $f_j$ is the parent of $f_i$; we denote this edge $f_if_j$ by $e_i$. 

We will construct a $1$-planar graph $G$ and a set $S\subseteq V(G)$ such that $G[S]$ is a subdivision of $G_0$.
Let $V_0$ be the set of vertices of $G_0$ incident to $f_0$. 
For each $i\in \{1,\dots ,s\}$, let $V_i$ be the set of vertices incident with $f_i$ and not incident with any face in $\{f_1,\dots,f_{i-1}\}$.
Let $E_0$ be the set of edges of $H$ incident to $f_0$ and not dual to edges in $T^*$, 
and for each $i\in \{1,\dots ,s\}$, let $E_i$ be the set of edges of $E(H)\setminus(E_0\cup E_1\cup \cdots \cup E_{i-1})$ that are incident with $f_i$ and not dual to the edge $f_if_j$ of $T^*$ with $j>i$.
Thus the sets $E_i$ partition $E(H)$, and for each $i\in \{1,\dots , s-1\}$ the edge of $H$ dual to $e_i$ is in $E_j$ where $e_i=f_if_j$ with $i<j$.

Iteratively construct $G$ as follows.
First, for each $e\in E_0$, subdivide the edge-segment of $G_0$ corresponding to $e$ once. 
Let $X_0$ be the set of all subdivision vertices introduced in this step, together with $V_0$.
Now, for each $v\in X_0$, and each $j\in \{0,1,\dots, s\}$, create $2c+4$ new vertices $a_{v,j,1},a_{v,j,2},\dots, a_{v,j,2c+4}$, and draw these vertices in $f_j$. Draw an edge from $v$ to each of these new vertices that does not cross any edge in  $H$ except those dual to the edges of the path from $f_0$ to $f_j$ in $T^*$.
For each pair of distinct vertices $v,w\in X_0$ and each $j\in \{1,\dots, 2c+4\}$, draw an edge from $v$ to $a_{w,0,j}$ in $f_0$.
At the final step, we will freely subdivide all edges incident to these newly created vertices, so the number of crossings on these edges is unimportant.

For $i=1,2,\dots,s$ apply the following step.
First, for each $e=pq\in E_i$, 
subdivide the edge-segment of $G_0$ corresponding to $e$ 
once between $p$ and the first crossing on $e$ starting at $p$,
once between $q$ and the first crossing on $e$ starting at $q$, and
once between each pair of consecutive crossings on $e$. 
No future edges will cross $e$.
Let $X_i$ be the the set of all subdivision vertices added to edge-segments corresponding to edges in $E_i$, together with all vertices in $V_i$.
For each $v\in X_i$, each $j\in \{i,i+1,\dots, s\}$ and each $y\in \{1,\dots ,2c+4\}$, create a new vertex $a_{v,j,y}$ and draw this vertex in $f_j$. Draw an edge from $v$ to each of these new vertices that does not cross any edge in  $H$ except 
those dual to edges of the path from $f_i$ to $f_j$ in $T^*$.
For each pair of distinct vertices $v\in X_i$ and $w\in X_0 \cup \dots \cup X_i$ 
and each $y\in \{1,\dots , 2c+4\}$, draw an edge from $v$ to $a_{w,i,y}$ in $f_i$.

Let $G'$ be the graph constructed so far.
Let $S$ be the set of all vertices in $V(G_0)$ together with all subdivision vertices introduced so far (that is, all vertices not of the form $a_{v,j,y}$). 
By construction, the subgraph of $G'$ induced by $S$ is a subdivision of $G_0$, and every edge in this subgraph is crossed at most once in the entire graph.
Subdivide the remaining edges (those incident to vertices of the form $a_{v,j}$ or $b_{v,j}$) until the resulting graph is $1$-planar, and call the final graph $G$.

We first construct a tree-decomposition of $G'$.
Let $T$ be the star with central vertex $r$ and leaf-set $V(G')\setminus S$.
Set $B'_r:=S$ and $B'_{w}:=N_{G'}[w]$ for each $w\in V(G')\setminus S$. Since $V(G')\setminus S$ is an independent set in $G'$, 
$(B'_x:x\in V(T))$ is a tree-decomposition of $G'$ with width $|S|$ (since 
$N_{G'}(w)\subseteq S$ and 
$|B'_{w}| =|N_{G'}[w]| \leq |S|+1$ for each $w\in V(G')\setminus S$). 
Since $G$ is a subdivision of $G'$,  $\tw(G) = \tw(G')\leq |S|$.

Now consider an arbitrary collection $\{\{a_1,b_1\},\dots,\{a_z,b_z\}$ of disjoint pairs of vertices in $S$.
By construction, for each $i\in \{1,\dots , z\}$, there are at least $2c+4$ common neighbours $a_{v,j,y}$ of $a_i$ and $b_i$ in $G'$ with $v\in \{a_i,b_i\}$.
Thus there is a collection $\{P_{i,j}:i\in \{1,\dots, z\},j\in \{1,\dots, 2c+4\}\}$ of internally disjoint paths in $G$, each internally disjoint from $S$, such that for each $i\in \{1,\dots, z\}$ and $j\in \{1,\dots, 2c+4\}$ the endvertices of $P_{i,j}$ are $a_i$ and $b_i$.
By \cref{lem:verywelllinked}, every tree-decomposition of $G$ of width less than $|S|+c$ (and thus  every tree-decomposition of $G$ of width less than $\tw(G)+c$) has a bag containing $S$.
\end{proof}

\section{\boldmath Width \texorpdfstring{$O(\sqrt{n})$}{O(√n)}}
\label{LTW}

This section constructs tree-decompositions with bags of bounded treewidth
in more general graph classes than those studied above, at the expense that the optimal width condition is relaxed to the asymptotically tight bound of $O(\sqrt{n})$ for $n$-vertex graphs. In fact, we show that the union of any bounded number of bags induces a subgraph with bounded treewidth. 


The following definitions are key ingredients to the proofs. A \defn{layering} of a graph $G$ is an ordered partition $(V_1,\dots,V_n)$ of $V(G)$ into (possibly empty) sets such that for each edge $vw\in E(G)$ there exists $i\in\{1,2,\dots,n-1\}$ such that $\{v,w\}\subseteq V_i\cup V_{i+1}$. The \defn{layered treewidth} of a graph $G$, denoted by \defn{$\ltw(G)$}, is the minimum nonnegative integer $\ell$ such that $G$ has a tree-decomposition $\mathcal{X} = (X_x:x\in V(T))$ and a layering $(V_1,\dots,V_n)$, such that $|X_x\cap V_i|\leq\ell$ for each bag $X_x$ and layer $V_i$. This implies that the subgraph induced by each layer has bounded treewidth, and moreover, a single tree-decomposition of $G$ has bounded treewidth when restricted to each layer. In fact, these properties hold when considering a bounded sequence of consecutive layers. Layered treewidth was independently introduced by \citet{DMW17} and \citet{Shahrokhi13}. 

The next lemma extends a result of Sergey Norin who proved the $2\sqrt{cn}$ treewidth bound~\citep{DMW17}. Here we choose to present a relatively simple proof rather than optimising the bound on the treewidth of the union of bags. 

\begin{lem}
\label{ltw}
    Let $G$ be a graph with $n$ vertices and layered treewidth $c\geq 1$. Then $G$ has  a tree-decomposition $\DD$ with width at most $2\sqrt{cn}$, such that the subgraph of $G$ induced by the union of any $k$ bags of $\DD$ has treewidth at most $(3k+1)c-1$. 
\end{lem}

\begin{proof}
    Let $(V_0,V_1,\dots,V_m)$ be a layering of $G$, and let $(B_x:x\in V(T))$ be a tree-decomposition of $G$, such that $|B_x \cap V_i|\leq c$ for each $x\in V(T)$ and $i\in\{0,1,\dots,m\}$. Let $p:=\ceil{\sqrt{n/c}}$. For $i\in\{0,1,\dots,p-1\}$ let $\widehat{V}_i:=\bigcup\{V_j:j\equiv i\pmod{p}\}$. So $\widehat{V}_0,\widehat{V}_1,\dots,\widehat{V}_{p-1}$ is a partition of $V(G)$, and $|\widehat{V}_i|\leq\frac{n}{p}$ for some $i\in\{0,1,\dots,p-1\}$. Each component of $G-\widehat{V}_i$ is contained within $p-1$ layers, and thus has treewidth at most $c(p-1)-1$. Hence $G-\widehat{V}_i$  has treewidth at most $c(p-1)-1$. Adding $\widehat{V}_i$ to every bag gives a tree-decomposition of $G$ with width at most $\frac{n}{p}+c(p-1)-1\leq 2\sqrt{cn}$. 
    
    Consider arbitrary bags $C_1,\dots,C_k$ of this tree-decomposition of $G$. By construction, there is a set $S\subseteq V(T)$ with $|S|\leq k$ such that $C_1\cup\dots\cup C_k\subseteq X:= \widehat{V}_i\cup \bigcup\{B_x:x\in S\}$.
    We now construct a tree-decomposition of $G[X]$ with width at most $(3k+1)c-1$, which implies that $\tw( G[C_1\cup\dots\cup C_k] )\leq (3k+1)c-1$ as desired. 

Fix a vertex $r$ of $T$. 
Let $T_j$ be a copy of $T$ for each $j\equiv i \pmod{p}$, where these copies are pairwise disjoint. 
For each $x\in V(T)$, let $x_j$ be the copy of $x$ in $T_j$. 
Let $P$ be the path $(y_0,y_1,\dots,y_m)$, where $y_j$ is identified with the vertex $r_j$ whenever $j\equiv i\pmod{p}$. 
We obtain a tree $U$.
For each node $x_j\in V(U)$, let 
\[A_{x_j}:= (B_x\cap V_j)\cup \bigcup\{ ( V_{j-1} \cup V_j \cup V_{j+1}) \cap B_z : z\in S\}.\] 
For each node $y_j\in V(U)$ with $j\not\equiv i\pmod{p}$, let 
\[A_{y_j}:= \bigcup\{ (V_{j-1} \cup V_j \cup V_{j+1} )\cap B_z : z\in S\}.\]
We now prove that $(A_u:u\in V(U))$ is a tree-decomposition of $G[X]$. 

Consider a vertex $v\in X$.  Say $v\in V_j$. 
If $v\in\bigcup\{B_z:z\in S\}$ and $j\bmod{p}$ is in $\{i-1,i,i+1\}$ then $v\in A_{x_j}$ for every $x\in V(T)$, and $v\in A_{y_{j-1}} \cup A_{y_{j+1}}$, and $v$ is in no other bag $A_u$ with $u\in V(U)$. 
If $v\in\bigcup\{B_z:z\in S\}$ and $j\bmod{p}$ is not in $\{i-1,i,i+1\}$ then $v\in A_{y_{j-1}} \cup A_{y_j} \cup A_{y_{j+1}}$, and $v$ is in no other bag $A_u$ with $u\in V(U)$. 
Otherwise, $v\not \in\bigcup\{B_z:z\in S\}$. Since $v\in X$, we have $j\equiv i \pmod{p}$, implying $v\in A_{x_j}$ whenever $v\in B_x$ with $x\in V(T)$, and $v$ is in no other bag $A_u$ with $u\in V(U)$. 
In each case, the set of bags $A_u$ that contain $v$ correspond to a connected subtree of $U$. 
Hence $(A_u:u\in V(U))$ satisfies the vertex-property of tree-decompositions. 

Consider an edge $vw$ of $G[X]$. 
So $v,w\in B_x$ for some node $x\in V(T)$. 
Say $v\in V_j$ and $w\in V_\ell$. 
If $j,\ell\equiv i \pmod{p}$ then $v,w\in A_{x_j}$. 
If $j\equiv i \pmod{p}$ and $\ell\not\equiv i \pmod{p}$, then $|j-\ell|=1$ and $w\in \cup\{ B_z:z\in S\}$, implying that $v,w\in A_{x_j}$. 
If $j\not\equiv i \pmod{p}$ and $\ell\not\equiv i \pmod{p}$, then $|j-\ell|\leq 1$ and $v,w\in \cup\{B_z:z\in S\}$, implying that $v,w\in A_{y_j}$. 
Hence $(A_u:u\in V(U))$ satisfies the edge-property of tree-decompositions. 

Therefore $(A_u:u\in V(U))$ is a tree-decomposition of $G[X]$. Observe that  $|A_{x_j}|\leq (3k+1)c $ and $|A_{y_j}|\leq 3ck$. Thus the width of $(A_u:u\in V(U))$ is at most $(3k+1)c-1$. Hence $G[X]$ and $G[C_1\cup\dots\cup C_k]$ have treewidth at most $(3k+1)c-1$.
\end{proof}

The next lemma shows the versatility of layered treewidth.

\begin{lem}
\label{ltw2}
Every graph $G$ with layered treewidth $c\geq 1$ has a set $S$ of at most $\sqrt{cn}$ vertices, such that $\tw(G[S])\leq c-1$ and $G-S$ has a tree-decomposition with width at most $\sqrt{cn}$ in which the union of any $k\geq 1$ bags has pathwidth at most $2ck-1$.
\end{lem}

\begin{proof}
Let $(V_0,V_1,\dots,V_m)$ be a layering of $G$, and let $(B_x:x\in V(T))$ be a tree-decomposition of $G$, such that $|B_x \cap V_i|\leq c$ for each $x\in V(T)$ and $i\in\{0,1,\dots,m\}$. Let $p:=\ceil{\sqrt{n/c}}$. 
Since $n>c$, we have $p\geq 2$. 
For $i\in\{0,1,\dots,p-1\}$, let $\widehat{V}_i:=\bigcup\{V_j:j\equiv i\pmod{p}\}$. So $\widehat{V}_0,\widehat{V}_1,\dots,\widehat{V}_{p-1}$ is a partition of $V(G)$, and $|\widehat{V}_i|\leq\frac{n}{p}$ for some $i\in\{0,1,\dots,p-1\}$. 
Let $S:=\widehat{V}_i$. 
Each component $C$ of $G-S$ is contained within $p-1$ layers, and thus 
$(B_x\cap V(C): x \in V(T))$ is  a tree-decomposition of $C$ with width at most $c(p-1)-1$. Hence $G-S$  has treewidth at most $c(p-1)-1\leq \sqrt{cn}$. 
Let $X$ be the union of $k$ bags from this tree-decomposition of $G-S$.
Then $X\subseteq \widehat{X}:=B_{x_1}\cup\dots\cup B_{x_k}$ for some $x_1,\dots,x_k\in V(T)$. 
Observe that 
$( \widehat{X} \cap (V_0\cup V_1), \widehat{X} \cap (V_1\cup V_2),\dots, \widehat{X}\cap (V_{m-1}\cup V_m) )$ is a path-decomposition of $G[\widehat{X}]$ with width at most $2ck-1$. 
The result follows. 
\end{proof}




\citet{DMW17} showed that every planar graph has layered treewidth at most 3. Thus \cref{ltw,ltw2} 
imply:

\begin{cor}
\label{PlanarSqrtn}
Every planar graph with $n$ vertices has:
\begin{enumerate}[(a)]
    \item a tree-decomposition $\DD$ with width at most $2\sqrt{3n}$, such that the subgraph of $G$ induced by the union of any $k$ bags of $\DD$ has treewidth at most $9k+2$, and
    \item a set $S$ of at most $\sqrt{3n}$ vertices, such that $\tw(G[S])\leq 2$ and $G-S$ has a tree-decomposition with width at most $\sqrt{3n}$ in which the union of any $k\geq 1$ bags has pathwidth at most $6k-1$.
\end{enumerate}
\end{cor}

\citet{DMW17} showed that every graph of Euler genus at most $g$ has layered treewidth at most $2g+3$. Thus \cref{ltw,ltw2} 
implies:

\begin{cor}
Every graph of Euler genus $g$ with $n$ vertices has: 
\begin{enumerate}[(a)]
\item a tree-decomposition $\DD$ with width at most $2\sqrt{(2g+3)n}$, such that the subgraph of $G$ induced by the union of any $k$ bags of $\DD$ has treewidth at most $(3k+1)(2g+3)-1$, and
\item a set $S$ of at most $\sqrt{(2g+3)n}$ vertices, such that $\tw(G[S])\leq 2g+2$ and $G-S$ has a tree-decomposition with width at most $\sqrt{(2g+3)n}$ in which the union of any $k\geq 1$ bags has pathwidth at most $2(2g+3)k-1$.
\end{enumerate}
\end{cor}

Layered treewidth is of interest beyond minor-closed classes, since there are several natural graph classes that have bounded layered treewidth but contain arbitrarily large complete graph minors~\citep{DEW17,HW24,HKW}. Here is one example. A graph $G$ is \defn{$(g,\ell)$-planar} if $G$ has a drawing in a surface of Euler genus at most $g$ with at most $\ell$ crossings on each edge. \citet{DEW17} showed that every $(g,\ell)$-planar graph has layered treewidth at most $(4g+6)(\ell+1)$. Thus \cref{ltw} implies the following result: 

\begin{cor}
\label{glPlanarSqrtn}
Every $(g,\ell)$-planar graph with $n$ vertices has:
\begin{enumerate}[(a)]
\item a tree-decomposition $\DD$ with width at most $2\sqrt{(2g+3)(\ell+1)n}$, such that the subgraph of $G$ induced by the union of any $k$ bags of $\DD$ has treewidth at most $(2k+1)(4g+6)(\ell+1)-1$, and
\item a set $S$ of at most $\sqrt{(4g+6)(\ell+1)n}$ vertices, such that $\tw(G[S])\leq (4g+6)(\ell+1)-1$ and $G-S$ has a tree-decomposition with width at most $\sqrt{(4g+6)(\ell+1)n}$ in which the union of any $k\geq 1$ bags has pathwidth at most $2(4g+6)(\ell+1)k-1$.
\end{enumerate}
\end{cor}

Note that the $g=0$ and $\ell=k=1$ case of \cref{glPlanarSqrtn} says that every 1-planar graph with $n$ vertices has a tree-decomposition with width at most $2\sqrt{6n}$, such that each bag has treewidth at most $35$, as promised in \cref{Intro}. This treewidth bound can be improved by optimising \cref{ltw} in the $k=1$ case. 

Now consider $K_p$-minor-free graphs where $p$ is a fixed positive integer. \citet{AST-SJDM94} showed that every $K_p$-minor-free graph $G$ with $n$ vertices has treewidth $O(\sqrt{n})$. \citet{LNW} showed that every $K_p$-minor-free graph $G$ has a tree-decomposition such that the subgraph of $G$ induced by any bag has bounded treewidth. The following theorem combines and generalises these two results. 

\begin{thm}
\label{KtMinorFree}
For any integer $p\geq 1$ there exists $c$ such that every $K_p$-minor-free graph $G$ with $n$ vertices has  a tree-decomposition $\DD$ with width at most $c\sqrt{n}$, such that the subgraph of $G$ induced by the union of any $k\geq 1$ bags of $\DD$ has treewidth at most $ck$.
\end{thm}

The proof of \cref{KtMinorFree} employs the following version of the graph minor structure theorem in terms of layered treewidth, due to \citet{DMW17}. If $(B_x:x\in V(T))$ is a tree-decomposition  
of a graph $G$, then for each node $x\in V(T)$, the \defn{torso} at $x$ is the graph \defn{$G\langle{B_x}\rangle$} obtained from $G[B_x]$ by adding edges so that $B_x\cap B_y$ is a clique for each edge $xy\in E(T)$.  

\begin{thm}[\citep{DMW17}]
\label{gmst-ltw}
For any integer $p\geq 1$ there exists $c$ such that every $K_p$-minor-free graph $G$ has a tree-decomposition $(B_x:x\in V(T))$ such that for each $x\in V(T)$ there exists $A_x\subseteq B_x$ with $|A_x|\leq c$ and $\ltw(G\langle{B_x}\rangle-A_x)\leq c$.
\end{thm}

\begin{proof}[Proof of \cref{KtMinorFree}]
   Let  $(B_x:x\in V(T))$ be the  tree-decomposition of $G$ from \cref{gmst-ltw}. For each $x\in V(T)$ there exists $A_x\subseteq B_x$ such that $|A_x|\leq c$ and $\ltw(G\langle{B_x}\rangle-A_x)\leq c$. By \cref{ltw}, 
   $G\langle{B_x}\rangle-A_x$  has a tree-decomposition $\DD_x$ with width at most $2\sqrt{cn}$, such that the subgraph of $G$ induced by the union of any $k$ bags of $\DD_x$ has treewidth at most $(3k+1)c-1$. Add $A_x$ to every bag of $\DD_x$. Now $\DD_x$ is  a tree-decomposition of $G\langle{B_x}\rangle$ with width at most $2\sqrt{cn}+c$, such that the subgraph of $G\langle{B_x}\rangle$ induced by the union of any $k\geq 1$ bags of $\DD_x$ has treewidth at most $(3k+2)c-1$. Let $\DD$ be the tree-decomposition of $G$ obtained as follows. For each edge $xy\in E(T)$, let $I_{xy}:=B_x\cap B_y$. Since $I_{xy}$ is a clique in  $G\langle{B_x}\rangle$, there is a bag $R_x$ in $\DD_x$ and there is a bag $R_y$ in $\DD_y$ such that $I_{xy}\subseteq R_x\cap R_y$.  Add an edge between (the nodes corresponding to the bags) $R_x$ and $R_y$, to obtain a tree-decomposition of $G$ with width at most $2\sqrt{cn}+c \leq c'\sqrt{n}$. 
   
   Let $\CC$ be a set of $k$ bags in this tree-decomposition. Let $X:=\bigcup \CC$. For each $x\in V(T)$ let $\CC_x:=\CC\cap \DD_x$. As shown above, for each $x\in V(T)$, the subgraph of $G\langle{B_x}\rangle$ induced by $\bigcup \CC_x$ has 
   a tree-decomposition $\TT_x$ with width at most $(3|\CC_x|+2)c-1\leq (3k+2)c-1$. 
   For each edge $xy\in E(T)$ let $I_{xy} := 
   (B_x \cap \bigcup \CC_x) \cap (B_y \cap \bigcup \CC_y)$, which is a clique in 
   $G\langle{B_x}\rangle$ and in $G\langle{B_y}\rangle$ (by the definition of torso). So there is a bag $R_x$ in $\TT_x$ and there is a bag $R_y$ in $\TT_y$ such that $I_{xy}\subseteq R_x\cap R_y$.  Add an edge between (the nodes corresponding to the bags) $R_x$ and $R_y$, to obtain a tree-decomposition of $\bigcup \CC$ with width at most $(3k+2)c-1$.
\end{proof}

\section{Open Problems}
\label{open}

We conclude by mentioning three open problems that arise from this work:

\begin{enumerate}

\item Recall that \cref{Degree3} says that every graph with maximum degree at most 3 has an optimal tree-decomposition such that every bag has treewidth at most 3. The following natural question arises: Is there a constant $c$ such that every graph with maximum degree at most 4 has an optimal tree-decomposition such that every bag has treewidth at most $c$? Note that \citet{LNW} showed that $\text{tree-tw}(G)\leq 15$ for every graph $G$ with maximum degree at most 4, and there is a constant $c$ such that $\text{tree-tw}(G)\leq c$ for every graph $G$ with maximum degree at most 5. This type of result does not hold for much larger maximum degree. In particular, \citet{LNW} showed that the class of 146-regular $n$-vertex graphs (for even $n$) has tree-tw $\Omega(n)$.

\item Are there constants $c,k\geq 1$ such that every 1-planar graph $G$ has a tree-decomposition of width at most $c\,\tw(G)$ such that each bag has treewidth at most $k$?

\item Can tree-decompositions with bags of small treewidth be used to speed up algorithms? Numerous problems can be solved on $n$-vertex graphs $G$ with time complexity $2^{O(\tw(G))}n$ via dynamic programming on a tree-decomposition (see \citep{CFKLMPPS15}). Can such results be improved using that each bag has bounded treewidth? 
\end{enumerate}

\subsection*{Acknowledgements} 

Thanks to Paul Wollan for helpful discussions on the Flat Wall Theorem.  Vida Dujmovi\'c independently observed that \cref{ltw} holds in the $k=1$ case. 

\fontsize{10pt}{11pt}
\selectfont


\begin{thebibliography}{79}
\providecommand{\natexlab}[1]{#1}
\providecommand{\msn}[1]{MR:\,\href{http://www.ams.org/mathscinet-getitem?mr=MR{#1}}{#1}}
\providecommand{\ZBL}[1]{Zbl:\,\href{https://www.zentralblatt-math.org/zmath/en/search/?q=an:#1}{#1}}
\providecommand{\url}[1]{\texttt{#1}}
\providecommand{\urlprefix}{}
\expandafter\ifx\csname urlstyle\endcsname\relax
  \providecommand{\doi}[1]{doi:\discretionary{}{}{}#1}\else
  \providecommand{\doi}{doi:\discretionary{}{}{}\begingroup
  \urlstyle{rm}\Url}\fi

\bibitem[{Abrishami et~al.(2024{\natexlab{a}})Abrishami, Alecu, Chudnovsky,
  Hajebi, and Spirkl}]{AACHS24}
\textsc{Tara Abrishami, Bogdan Alecu, Maria Chudnovsky, Sepehr Hajebi, and
  Sophie Spirkl}.
\newblock \href{https://doi.org/10.1016/j.jctb.2023.10.008}{Induced subgraphs
  and tree decompositions {VII}. {B}asic obstructions in {$H$}-free graphs}.
\newblock \emph{J. Combin. Theory Ser. B}, 164:443--472, 2024{\natexlab{a}}.

\bibitem[{Abrishami et~al.(2024{\natexlab{b}})Abrishami, Alecu, Chudnovsky,
  Hajebi, and Spirkl}]{AACHS24a}
\textsc{Tara Abrishami, Bogdan Alecu, Maria Chudnovsky, Sepehr Hajebi, and
  Sophie Spirkl}.
\newblock \href{https://doi.org/10.1007/s00493-024-00097-0}{Induced subgraphs
  and tree decompositions {VIII:} {E}xcluding a forest in ({T}heta,
  {P}rism)-free graphs}.
\newblock \emph{Combinatorica}, 44(5):921--948, 2024{\natexlab{b}}.

\bibitem[{Abrishami et~al.(2025{\natexlab{a}})Abrishami, Alecu, Chudnovsky,
  Hajebi, and Spirkl}]{AACHS25}
\textsc{Tara Abrishami, Bogdan Alecu, Maria Chudnovsky, Sepehr Hajebi, and
  Sophie Spirkl}.
\newblock \href{https://arxiv.org/abs/2307.13684}{Induced subgraphs and tree
  decompositions {X}. {T}owards logarithmic treewidth for even-hole-free
  graphs}.
\newblock 2025{\natexlab{a}}, arXiv:2307.13684.

\bibitem[{Abrishami et~al.(2024{\natexlab{c}})Abrishami, Alecu, Chudnovsky,
  Hajebi, Spirkl, and Vušković}]{AACHSV24}
\textsc{Tara Abrishami, Bogdan Alecu, Maria Chudnovsky, Sepehr Hajebi, Sophie
  Spirkl, and Kristina Vušković}.
\newblock \href{https://doi.org/10.1002/jgt.23055}{Induced subgraphs and tree
  decompositions {V}. {O}ne neighbor in a hole}.
\newblock \emph{J. Graph Theory}, 105(4):542--561, 2024{\natexlab{c}}.

\bibitem[{Abrishami et~al.(2024{\natexlab{d}})Abrishami, Alecu, Chudnovsky,
  Hajebi, Spirkl, and Vušković}]{AACHSV24a}
\textsc{Tara Abrishami, Bogdan Alecu, Maria Chudnovsky, Sepehr Hajebi, Sophie
  Spirkl, and Kristina Vušković}.
\newblock \href{https://doi.org/10.1002/jgt.23104}{Tree independence number
  {I}. ({E}ven hole, diamond, pyramid)-free graphs}.
\newblock \emph{J. Graph Theory}, 106(4):923--943, 2024{\natexlab{d}}.

\bibitem[{Abrishami et~al.(2024{\natexlab{e}})Abrishami, Chudnovsky, Dibek,
  Hajebi, Rzążewski, Spirkl, and Vušković}]{ACDHRSV24}
\textsc{Tara Abrishami, Maria Chudnovsky, Cemil Dibek, Sepehr Hajebi, Paweł
  Rzążewski, Sophie Spirkl, and Kristina Vušković}.
\newblock \href{https://doi.org/10.1016/j.jctb.2023.10.005}{Induced subgraphs
  and tree decompositions {II}. {T}oward walls and their line graphs in graphs
  of bounded degree}.
\newblock \emph{J. Combin. Theory Ser. B}, 164:371--403, 2024{\natexlab{e}}.

\bibitem[{Abrishami et~al.(2022{\natexlab{a}})Abrishami, Chudnovsky, Hajebi,
  and Spirkl}]{ACHS22}
\textsc{Tara Abrishami, Maria Chudnovsky, Sepehr Hajebi, and Sophie Spirkl}.
\newblock \href{https://doi.org/10.19086/aic.2022.6}{Induced subgraphs and tree
  decompositions {III}. {T}hree-path-configurations and logarithmic treewidth}.
\newblock \emph{Adv. Comb.}, \#6, 2022{\natexlab{a}}.

\bibitem[{Abrishami et~al.(2023)Abrishami, Chudnovsky, Hajebi, and
  Spirkl}]{ACHS23}
\textsc{Tara Abrishami, Maria Chudnovsky, Sepehr Hajebi, and Sophie Spirkl}.
\newblock \href{https://doi.org/10.37236/11623}{Induced subgraphs and tree
  decompositions {IV}. ({E}ven hole, diamond, pyramid)-free graphs}.
\newblock \emph{Electron. J. Combin.}, 30(2), 2023.

\bibitem[{Abrishami et~al.(2025{\natexlab{b}})Abrishami, Chudnovsky, Hajebi,
  and Spirkl}]{ACHS25}
\textsc{Tara Abrishami, Maria Chudnovsky, Sepehr Hajebi, and Sophie Spirkl}.
\newblock \href{https://doi.org/10.1016/j.disc.2024.114195}{Induced subgraphs
  and tree decompositions {VI.} {G}raphs with 2-cutsets}.
\newblock \emph{Discrete Math.}, 348(1):114195, 2025{\natexlab{b}}.

\bibitem[{Abrishami et~al.(2022{\natexlab{b}})Abrishami, Chudnovsky, and
  Vu\v{s}kovi\'{c}}]{ACV22}
\textsc{Tara Abrishami, Maria Chudnovsky, and Kristina Vu\v{s}kovi\'{c}}.
\newblock \href{https://doi.org/10.1016/j.jctb.2022.05.009}{Induced subgraphs
  and tree decompositions {I}. {E}ven-hole-free graphs of bounded degree}.
\newblock \emph{J. Combin. Theory Ser. B}, 157:144--175, 2022{\natexlab{b}}.

\bibitem[{Adler(2006)}]{Adler06}
\href{https://freidok.uni-freiburg.de/data/2468}{Width functions for hypertree
  decompositions}.
\newblock \textsc{Isolde Adler},
  \href{https://freidok.uni-freiburg.de/data/2468}{Width functions for
  hypertree decompositions}.
\newblock Ph.D. thesis, Univ. Freiburg, 2006.

\bibitem[{Alecu et~al.(2024)Alecu, Chudnovsky, Hajebi, and Spirkl}]{ACHS24}
\textsc{Bogdan Alecu, Maria Chudnovsky, Sepehr Hajebi, and Sophie Spirkl}.
\newblock \href{http://dx.doi.org/10.19086/aic.2024.6}{Induced subgraphs and
  tree decompositions {XIII}. {B}asic obstructions in $\mathcal{H}$-free graphs
  for finite $\mathcal{H}$}.
\newblock \emph{Advances in Combinatorics}, \#6, 2024.

\bibitem[{Alecu et~al.(2025{\natexlab{a}})Alecu, Chudnovsky, Hajebi, and
  Spirkl}]{ACHS25a}
\textsc{Bogdan Alecu, Maria Chudnovsky, Sepehr Hajebi, and Sophie Spirkl}.
\newblock \href{https://arxiv.org/abs/2309.04390}{Induced subgraphs and tree
  decompositions {XI}. {L}ocal structure in even-hole-free graphs of large
  treewidth}.
\newblock 2025{\natexlab{a}}, arXiv:2309.04390.

\bibitem[{Alecu et~al.(2025{\natexlab{b}})Alecu, Chudnovsky, Hajebi, and
  Spirkl}]{ACHS25b}
\textsc{Bogdan Alecu, Maria Chudnovsky, Sepehr Hajebi, and Sophie Spirkl}.
\newblock \href{http://dx.doi.org/10.5802/igt.6}{Induced subgraphs and tree
  decompositions {XII}. {G}rid theorem for pinched graphs}.
\newblock \emph{Innovations in Graph Theory}, 2:1--23, 2025{\natexlab{b}}.

\bibitem[{Alecu et~al.(2025{\natexlab{c}})Alecu, Chudnovsky, and
  Spirkl}]{ACS25}
\textsc{Bogdan Alecu, Maria Chudnovsky, and Sophie Spirkl}.
\newblock \href{http://dx.doi.org/10.19086/aic.2025.3}{Induced subgraphs and
  tree decompositions {IX}. {G}rid theorem for perforated graphs}.
\newblock \emph{Advances in Combinatorics}, \#3, 2025{\natexlab{c}}.

\bibitem[{Alon et~al.(2025)Alon, Milani\v{c}, and Rz\k{a}\.z{}ewski}]{AMR25}
\textsc{Noga Alon, Martin Milani\v{c}, and Pawe\l{} Rz\k{a}\.z{}ewski}.
\newblock \href{https://arxiv.org/abs/2511.03864}{Induced matching treewidth
  and tree-independence number, revisited}.
\newblock 2025, arXiv:2511.03864.

\bibitem[{Alon et~al.(1994)Alon, Seymour, and Thomas}]{AST-SJDM94}
\textsc{Noga Alon, Paul Seymour, and Robin Thomas}.
\newblock \href{https://doi.org/10.1137/S0895480191198768}{Planar separators}.
\newblock \emph{SIAM J. Discrete Math.}, 7(2):184--193, 1994.

\bibitem[{Barrera-Cruz et~al.(2019)Barrera-Cruz, Felsner, M\'{e}sz\'{a}ros,
  Micek, Smith, Taylor, and Trotter}]{BFMMSTT19}
\textsc{Fidel Barrera-Cruz, Stefan Felsner, Tam\'{a}s M\'{e}sz\'{a}ros, Piotr
  Micek, Heather Smith, Libby Taylor, and William~T. Trotter}.
\newblock \href{https://doi.org/10.1016/j.jctb.2019.02.003}{Separating
  tree-chromatic number from path-chromatic number}.
\newblock \emph{J. Combin. Theory Ser. B}, 138:206--218, 2019.

\bibitem[{Berger and Seymour(2024)}]{BS24}
\textsc{Eli Berger and Paul Seymour}.
\newblock \href{https://doi.org/10.1007/s00493-024-00088-1}{Bounded-diameter
  tree-decompositions}.
\newblock \emph{Combinatorica}, 44(3):659--674, 2024.

\bibitem[{Bodlaender(1998)}]{Bodlaender98}
\textsc{Hans~L. Bodlaender}.
\newblock \href{https://doi.org/10.1016/S0304-3975(97)00228-4}{A partial
  $k$-arboretum of graphs with bounded treewidth}.
\newblock \emph{Theoret. Comput. Sci.}, 209(1-2):1--45, 1998.

\bibitem[{Cames~van Batenburg et~al.(2019)Cames~van Batenburg, Huynh, Joret,
  and Raymond}]{CvBHJR19}
\textsc{Wouter Cames~van Batenburg, Tony Huynh, Gwena\"el Joret, and
  Jean-Florent Raymond}.
\newblock \href{https://arxiv.org/abs/1807.04969}{A tight
  {E}rd{\H{o}}s-{P}\'osa function for planar minors}.
\newblock \emph{Adv. Comb.}, \#2, 2019.

\bibitem[{Chekuri and Chuzhoy(2016)}]{CC16}
\textsc{Chandra Chekuri and Julia Chuzhoy}.
\newblock \href{https://doi.org/10.1145/2820609}{Polynomial bounds for the
  grid-minor theorem}.
\newblock \emph{J. ACM}, 63(5):40, 2016.

\bibitem[{Choi et~al.(2025)Choi, Hilaire, Milanič, and Wiederrecht}]{CHMW25}
\textsc{Mujin Choi, Claire Hilaire, Martin Milanič, and Sebastian
  Wiederrecht}.
\newblock \href{https://arxiv.org/abs/2506.08829}{Excluding an induced wheel
  minor in graphs without large induced stars}.
\newblock 2025, arXiv:2506.08829.

\bibitem[{Chudnovsky and Codsi(2025)}]{CC25}
\textsc{Maria Chudnovsky and Julien Codsi}.
\newblock \href{https://doi.org/10.48550/ARXIV.2509.15458}{Tree-independence
  number {VI}. {T}hetas and pyramids}.
\newblock 2025, arXiv:2509.15458.

\bibitem[{Chudnovsky et~al.(2025{\natexlab{a}})Chudnovsky, Codsi, Hajebi, and
  Spirkl}]{CCHS25}
\textsc{Maria Chudnovsky, Julien Codsi, Sepehr Hajebi, and Sophie Spirkl}.
\newblock \href{https://arxiv.org/abs/2506.05602}{Induced subgraphs and tree
  decompositions {XIX}. {T}hetas and forests}.
\newblock 2025{\natexlab{a}}, arXiv:2506.05602.

\bibitem[{Chudnovsky et~al.(2025{\natexlab{b}})Chudnovsky, Codsi, Lokshtanov,
  Milanic, and Sivashankar}]{CCLMS25}
\textsc{Maria Chudnovsky, Julien Codsi, Daniel Lokshtanov, Martin Milanic, and
  Varun Sivashankar}.
\newblock \href{https://doi.org/10.48550/arXiv.2501.14658}{Tree independence
  number {V}. {W}alls and claws}.
\newblock 2025{\natexlab{b}}, arXiv:2501.14658.

\bibitem[{Chudnovsky et~al.(2024{\natexlab{a}})Chudnovsky, Gartland, Hajebi,
  Lokshtanov, and Spirkl}]{CGHLS24}
\textsc{Maria Chudnovsky, Peter Gartland, Sepehr Hajebi, Daniel Lokshtanov, and
  Sophie Spirkl}.
\newblock \href{https://arxiv.org/abs/2402.14211}{Induced subgraphs and tree
  decompositions {XV}. {E}ven-hole-free graphs with bounded clique number have
  logarithmic treewidth}.
\newblock 2024{\natexlab{a}}, arXiv:2402.14211.

\bibitem[{Chudnovsky et~al.(2025{\natexlab{c}})Chudnovsky, Gartland, Hajebi,
  Lokshtanov, and Spirkl}]{CGHLS25}
\textsc{Maria Chudnovsky, Peter Gartland, Sepehr Hajebi, Daniel Lokshtanov, and
  Sophie Spirkl}.
\newblock \href{https://doi.org/10.1137/1.9781611978322.151}{Tree independence
  number {IV.} {E}ven-hole-free graphs}.
\newblock In \textsc{Yossi Azar and Debmalya Panigrahi}, eds., \emph{Proc. 2025
  Annual {ACM-SIAM} Symposium on Discrete Algorithms \textup{(SODA '25)}}, pp.
  4444--4461. {SIAM}, 2025{\natexlab{c}}.

\bibitem[{Chudnovsky et~al.(2026{\natexlab{a}})Chudnovsky, Hajebi, Lokshtanov,
  and Spirkl}]{CHLS26}
\textsc{Maria Chudnovsky, Sepehr Hajebi, Daniel Lokshtanov, and Sophie Spirkl}.
\newblock \href{https://doi.org/10.1016/j.jctb.2025.08.003}{Tree independence
  number {II.} {T}hree-path-configurations}.
\newblock \emph{J. Combin. Theory Ser. B}, 176:74--96, 2026{\natexlab{a}}.

\bibitem[{Chudnovsky et~al.(2024{\natexlab{b}})Chudnovsky, Hajebi, and
  Spirkl}]{CHS24}
\textsc{Maria Chudnovsky, Sepehr Hajebi, and Sophie Spirkl}.
\newblock \href{https://arxiv.org/abs/2411.11842}{Induced subgraphs and tree
  decompositions {XVII}. {A}nticomplete sets of large treewidth}.
\newblock 2024{\natexlab{b}}, arXiv:2411.11842.

\bibitem[{Chudnovsky et~al.(2024{\natexlab{c}})Chudnovsky, Hajebi, and
  Spirkl}]{CHS24a}
\textsc{Maria Chudnovsky, Sepehr Hajebi, and Sophie Spirkl}.
\newblock \href{https://arxiv.org/abs/2412.17756}{Induced subgraphs and tree
  decompositions {XVIII}. {O}bstructions to bounded pathwidth}.
\newblock 2024{\natexlab{c}}, arXiv:2412.17756.

\bibitem[{Chudnovsky et~al.(2026{\natexlab{b}})Chudnovsky, Hajebi, and
  Spirkl}]{CHS26}
\textsc{Maria Chudnovsky, Sepehr Hajebi, and Sophie Spirkl}.
\newblock \href{https://doi.org/10.1016/j.jctb.2025.09.005}{Induced subgraphs
  and tree decompositions {XVI.} {C}omplete bipartite induced minors}.
\newblock \emph{J. Combin. Theory {B}}, 176:287--318, 2026{\natexlab{b}}.

\bibitem[{Chudnovsky et~al.(2024{\natexlab{d}})Chudnovsky, Hajebi, and
  Trotignon}]{CHT24}
\textsc{Maria Chudnovsky, Sepehr Hajebi, and Nicolas Trotignon}.
\newblock \href{https://arxiv.org/abs/2406.13053}{Tree independence number
  {III}. {T}hetas, prisms and stars}.
\newblock 2024{\natexlab{d}}, arXiv:2406.13053.

\bibitem[{Chuzhoy(2015)}]{Chuzhoy15}
\textsc{Julia Chuzhoy}.
\newblock \href{https://doi.org/10.1145/2746539.2746551}{Excluded grid theorem:
  Improved and simplified}.
\newblock In \emph{Proc 48th Annual ACM on Symp. Theory of Computing
  \textup{(STOC '15)}}, pp. 645--654. ACM, 2015.

\bibitem[{Chuzhoy and Tan(2021)}]{CT21}
\textsc{Julia Chuzhoy and Zihan Tan}.
\newblock \href{https://doi.org/10.1016/j.jctb.2020.09.010}{Towards tight(er)
  bounds for the excluded grid theorem}.
\newblock \emph{J. Combin. Theory Ser. {B}}, 146:219--265, 2021.

\bibitem[{Courcelle(1990)}]{Courcelle90}
\textsc{Bruno Courcelle}.
\newblock \href{https://doi.org/10.1016/0890-5401(90)90043-H}{The monadic
  second-order logic of graphs. {I}. {R}ecognizable sets of finite graphs}.
\newblock \emph{Inform. and Comput.}, 85(1):12--75, 1990.

\bibitem[{Cygan et~al.(2015)Cygan, Fomin, Kowalik, Lokshtanov, Marx, Pilipczuk,
  Pilipczuk, and Saurabh}]{CFKLMPPS15}
\textsc{Marek Cygan, Fedor~V. Fomin, Lukasz Kowalik, Daniel Lokshtanov,
  D{\'{a}}niel Marx, Marcin Pilipczuk, Michal Pilipczuk, and Saket Saurabh}.
\newblock \href{https://doi.org/10.1007/978-3-319-21275-3}{Parameterized
  algorithms}.
\newblock Springer, 2015.

\bibitem[{Dallard et~al.(2024{\natexlab{a}})Dallard, Fomin, Golovach, Korhonen,
  and Milani\v{c}}]{DFGKM24}
\textsc{Cl{\'{e}}ment Dallard, Fedor~V. Fomin, Petr~A. Golovach, Tuukka
  Korhonen, and Martin Milani\v{c}}.
\newblock \href{https://doi.org/10.4230/LIPIcs.ICALP.2024.51}{Computing tree
  decompositions with small independence number}.
\newblock In \emph{Proc. 51st Int'l Coll. Automata, Languages, and Programming
  \textup{(ICALP '24)}}, vol. 297 of \emph{LIPIcs}, pp. 51:1--51:18. Schloss
  Dagstuhl, 2024{\natexlab{a}}.

\bibitem[{Dallard et~al.(2025)Dallard, Milanic, Munaro, and Yang}]{DMMY25}
\textsc{Cl{\'{e}}ment Dallard, Martin Milanic, Andrea Munaro, and Shizhou
  Yang}.
\newblock \href{https://doi.org/10.48550/ARXIV.2506.12424}{Layered
  tree-independence number and clique-based separators}.
\newblock 2025, arXiv:2506.12424.

\bibitem[{Dallard et~al.(2021)Dallard, Milani\v{c}, and \v{S}torgel}]{DMS21}
\textsc{Cl\'{e}ment Dallard, Martin Milani\v{c}, and Kenny \v{S}torgel}.
\newblock \href{https://doi.org/10.1137/20M1352119}{Treewidth versus clique
  number. {I}. {G}raph classes with a forbidden structure}.
\newblock \emph{SIAM J. Discrete Math.}, 35(4):2618--2646, 2021.

\bibitem[{Dallard et~al.(2024{\natexlab{b}})Dallard, Milani\v{c}, and
  \v{S}torgel}]{DMS24a}
\textsc{Cl\'{e}ment Dallard, Martin Milani\v{c}, and Kenny \v{S}torgel}.
\newblock \href{https://doi.org/10.1016/j.jctb.2023.10.006}{Treewidth versus
  clique number. {II}. {T}ree-independence number}.
\newblock \emph{J. Combin. Theory Ser. B}, 164:404--442, 2024{\natexlab{b}}.

\bibitem[{Dallard et~al.(2024{\natexlab{c}})Dallard, Milani\v{c}, and
  \v{S}torgel}]{DMS24b}
\textsc{Cl\'{e}ment Dallard, Martin Milani\v{c}, and Kenny \v{S}torgel}.
\newblock \href{https://doi.org/10.1016/j.jctb.2024.03.005}{Treewidth versus
  clique number. {III}. {T}ree-independence number of graphs with a forbidden
  structure}.
\newblock \emph{J. Combin. Theory Ser. B}, 167:338--391, 2024{\natexlab{c}}.

\bibitem[{Dallard et~al.(2024{\natexlab{d}})Dallard, Krnc, Kwon, Milani\v{c},
  Munaro, \v{S}torgel, and Wiederrecht}]{DKKMMSW24}
\textsc{Clément Dallard, Matjaž Krnc, O-joung Kwon, Martin Milani\v{c},
  Andrea Munaro, Kenny \v{S}torgel, and Sebastian Wiederrecht}.
\newblock \href{http://arxiv.org/abs/2402.11222}{Treewidth versus clique
  number. {IV.} {T}ree-independence number of graphs excluding an induced
  star}.
\newblock 2024{\natexlab{d}}, arXiv:2402.11222.

\bibitem[{Dehkordi and Farr(2021)}]{DF21}
\textsc{Hooman~Reisi Dehkordi and Graham Farr}.
\newblock \href{https://doi.org/10.37236/8816}{Non-separating planar graphs}.
\newblock \emph{Electron. J. Combin.}, 28(1):1, 2021.

\bibitem[{Dourisboure and Gavoille(2007)}]{DG07}
\textsc{Yon Dourisboure and Cyril Gavoille}.
\newblock \href{https://doi.org/10.1016/j.disc.2005.12.060}{Tree-decompositions
  with bags of small diameter}.
\newblock \emph{Discrete Math.}, 307(16):2008--2029, 2007.

\bibitem[{Dragan(2025)}]{Dragan25}
\textsc{Feodor~F. Dragan}.
\newblock \href{https://arxiv.org/abs/2502.00951}{Graph parameters that are
  coarsely equivalent to tree-length}.
\newblock 2025, arXiv:2502.00951.

\bibitem[{Dragan and Köhler(2025)}]{DK25}
\textsc{Feodor~F. Dragan and Ekkehard Köhler}.
\newblock \href{https://arxiv.org/abs/2503.05661}{Graph parameters that are
  coarsely equivalent to path-length}.
\newblock 2025, arXiv:2503.05661.

\bibitem[{Dujmovi\'c et~al.(2017)Dujmovi\'c, Eppstein, and Wood}]{DEW17}
\textsc{Vida Dujmovi\'c, David Eppstein, and David~R. Wood}.
\newblock \href{https://doi.org/10.1137/16M1062879}{Structure of graphs with
  locally restricted crossings}.
\newblock \emph{SIAM J. Discrete Math.}, 31(2):805--824, 2017.

\bibitem[{Dujmovi{\'c} et~al.(2017)Dujmovi{\'c}, Morin, and Wood}]{DMW17}
\textsc{Vida Dujmovi{\'c}, Pat Morin, and David~R. Wood}.
\newblock \href{https://doi.org/10.1016/j.jctb.2017.05.006}{Layered separators
  in minor-closed graph classes with applications}.
\newblock \emph{J. Combin. Theory Ser. B}, 127:111--147, 2017.
\newblock arXiv:1306.1595.

\bibitem[{Geelen et~al.(2004)Geelen, Richter, and Salazar}]{GRS04}
\textsc{Jim~F. Geelen, R.~Bruce Richter, and Gelasio Salazar}.
\newblock \href{https://doi.org/10.1016/j.ejc.2003.07.007}{Embedding grids in
  surfaces}.
\newblock \emph{European J. Combin.}, 25(6):785--792, 2004.

\bibitem[{Gorsky et~al.(2025)Gorsky, Seweryn, and Wiederrecht}]{GSW25}
\textsc{Maximilian Gorsky, Michał~T. Seweryn, and Sebastian Wiederrecht}.
\newblock \href{https://arxiv.org/abs/2504.02532}{Polynomial bounds for the
  graph minor structure theorem}.
\newblock 2025, arXiv:2504.02532.

\bibitem[{Harvey and Wood(2017)}]{HW17}
\textsc{Daniel~J. Harvey and David~R. Wood}.
\newblock \href{https://doi.org/10.1002/jgt.22030}{Parameters tied to
  treewidth}.
\newblock \emph{J. Graph Theory}, 84(4):364--385, 2017.

\bibitem[{Hendrey et~al.(2025{\natexlab{a}})Hendrey, Hickingbotham, Hodor, and
  Wood}]{HHHW}
\textsc{Kevin Hendrey, Robert Hickingbotham, Jędrzej Hodor, and David~R.
  Wood}.
\newblock Optimal tree-decompositions with bags of bounded pathwidth.
\newblock 2025{\natexlab{a}}.

\bibitem[{Hendrey et~al.(2025{\natexlab{b}})Hendrey, Karol, and Wood}]{HKW}
\textsc{Kevin Hendrey, Nikolai Karol, and David~R. Wood}.
\newblock \href{http://arxiv.org/abs/2507.22395}{Structure of
  $k$-matching-planar graphs}.
\newblock 2025{\natexlab{b}}, arXiv:2507.22395.

\bibitem[{Hickingbotham(2023)}]{Hickingbotham23}
\textsc{Robert Hickingbotham}.
\newblock \href{https://doi.org/10.37236/11364}{Induced subgraphs and path
  decompositions}.
\newblock \emph{Electron. J. Combin.}, 30(2), 2023.

\bibitem[{Hickingbotham(2025)}]{Hickingbotham25}
\textsc{Robert Hickingbotham}.
\newblock \href{https://arxiv.org/abs/2501.10840}{Graphs quasi-isometric to
  graphs with bounded treewidth}.
\newblock 2025, arXiv:2501.10840.

\bibitem[{Hickingbotham and Wood(2024)}]{HW24}
\textsc{Robert Hickingbotham and David~R. Wood}.
\newblock \href{https://doi.org/10.1137/22M1540296}{Shallow minors, graph
  products and beyond-planar graphs}.
\newblock \emph{SIAM J. Discrete Math.}, 38(1):1057--1089, 2024.

\bibitem[{Huynh and Kim(2017)}]{HK17}
\textsc{Tony Huynh and Ringi Kim}.
\newblock \href{https://doi.org/10.1002/jgt.22121}{Tree-chromatic number is not
  equal to path-chromatic number}.
\newblock \emph{J. Graph Theory}, 86(2):213--222, 2017.

\bibitem[{Huynh et~al.(2021)Huynh, Reed, Wood, and Yepremyan}]{HRWY21}
\textsc{Tony Huynh, Bruce Reed, David~R. Wood, and Liana Yepremyan}.
\newblock \href{https://doi.org/10.1007/978-3-030-62497-2_30}{Notes on tree-
  and path-chromatic number}.
\newblock In \emph{2019--20 {MATRIX} Annals}, vol.~4 of \emph{MATRIX Book
  Ser.}, pp. 489--498. Springer, 2021.

\bibitem[{Kawarabayashi et~al.(2018)Kawarabayashi, Thomas, and Wollan}]{KTW18}
\textsc{Ken-ichi Kawarabayashi, Robin Thomas, and Paul Wollan}.
\newblock \href{https://doi.org/10.1016/j.jctb.2017.09.006}{A new proof of the
  flat wall theorem}.
\newblock \emph{J. Combin. Theory Ser. B}, 129:204--238, 2018.

\bibitem[{Kobourov et~al.(2017)Kobourov, Liotta, and Montecchiani}]{KLM17}
\textsc{Stephen~G. Kobourov, Giuseppe Liotta, and Fabrizio Montecchiani}.
\newblock \href{https://doi.org/10.1016/j.cosrev.2017.06.002}{An annotated
  bibliography on 1-planarity}.
\newblock \emph{Comput. Sci. Rev.}, 25:49--67, 2017.

\bibitem[{Krause et~al.(2025)Krause, Redzic, and Ueckerdt}]{KRU25}
\textsc{Kilian Krause, Mirza Redzic, and Torsten Ueckerdt}.
\newblock \href{https://arxiv.org/abs/2504.19751}{On the relation between
  treewidth, tree-independence number, and tree-chromatic number of graphs}.
\newblock 2025, arXiv:2504.19751.

\bibitem[{Leaf and Seymour(2015)}]{LeafSeymour15}
\textsc{Alexander Leaf and Paul Seymour}.
\newblock \href{https://doi.org/10.1016/j.jctb.2014.09.003}{Tree-width and
  planar minors}.
\newblock \emph{J. Combin. Theory Ser. {B}}, 111:38--53, 2015.

\bibitem[{Lima et~al.(2024)Lima, Milani\v{c}, Muršič, Okrasa,
  Rz\k{a}\.{z}ewsk, and \v{S}torgel}]{LMMORS24}
\textsc{Paloma~T. Lima, Martin Milani\v{c}, Peter Muršič, Karolina Okrasa,
  Pawe\l{} Rz\k{a}\.{z}ewsk, and Kenny \v{S}torgel}.
\newblock \href{https://arxiv.org/abs/2402.15834}{Tree decompositions meet
  induced matchings: beyond max weight independent set}.
\newblock 2024, arXiv:2402.15834.

\bibitem[{Liu et~al.(2024)Liu, Norin, and Wood}]{LNW}
\textsc{Chun-Hung Liu, Sergey Norin, and David~R. Wood}.
\newblock \href{https://arxiv.org/abs/2410.20333}{Product structure and tree
  decompositions}.
\newblock 2024, arXiv:2410.20333.

\bibitem[{Liu and Yoo(2025)}]{LY25}
\textsc{Chun-Hung Liu and Youngho Yoo}.
\newblock \href{https://arxiv.org/abs/2509.09895}{Tree-width of a graph
  excluding an apex-forest or a wheel as a minor}.
\newblock 2025, arXiv:2509.09895.

\bibitem[{Malni\v{c} and Mohar(1992)}]{MM92}
\textsc{Aleksander Malni\v{c} and Bojan Mohar}.
\newblock \href{https://doi.org/10.1016/0095-8956(92)90015-P}{Generating
  locally cyclic triangulations of surfaces}.
\newblock \emph{J. Combin. Theory Ser. B}, 56(2):147--164, 1992.

\bibitem[{Milani\v{c} and Rz\k{a}\.{z}ewski(2022)}]{MR22}
\textsc{Martin Milani\v{c} and Pawe\l{} Rz\k{a}\.{z}ewski}.
\newblock \href{http://arxiv.org/abs/2209.12315}{Tree decompositions with
  bounded independence number: beyond independent sets}.
\newblock 2022, arXiv:2209.12315.

\bibitem[{Mohar and Thomassen(2001)}]{MoharThom}
\textsc{Bojan Mohar and Carsten Thomassen}.
\newblock Graphs on surfaces.
\newblock Johns Hopkins University Press, 2001.

\bibitem[{Nguyen et~al.(2025)Nguyen, Scott, and Seymour}]{NSS25}
\textsc{Tung Nguyen, Alex Scott, and Paul Seymour}.
\newblock \href{https://arxiv.org/abs/2501.09839}{Asymptotic structure. {I}.
  {C}oarse tree-width}.
\newblock 2025, arXiv:2501.09839.

\bibitem[{Reed(1999)}]{Reed99a}
\textsc{Bruce Reed}.
\newblock \href{https://doi.org/10.1007/s004930050056}{Mangoes and
  blueberries}.
\newblock \emph{Combinatorica}, 19(2):267--296, 1999.

\bibitem[{Reed(1997)}]{Reed97}
\textsc{Bruce~A. Reed}.
\newblock \href{https://doi.org/10.1017/CBO9780511662119.006}{Tree width and
  tangles: a new connectivity measure and some applications}.
\newblock In \textsc{R.~A. Bailey}, ed., \emph{Surveys in Combinatorics}, vol.
  241 of \emph{London Math. Soc. Lecture Note Ser.}, pp. 87--162. Cambridge
  Univ. Press, 1997.

\bibitem[{Robertson and Seymour(1986{\natexlab{a}})}]{RS-II}
\textsc{Neil Robertson and Paul Seymour}.
\newblock \href{https://doi.org/10.1016/0196-6774(86)90023-4}{Graph minors.
  {II}. {A}lgorithmic aspects of tree-width}.
\newblock \emph{J. Algorithms}, 7(3):309--322, 1986{\natexlab{a}}.

\bibitem[{Robertson and Seymour(1986{\natexlab{b}})}]{RS-V}
\textsc{Neil Robertson and Paul Seymour}.
\newblock \href{https://doi.org/10.1016/0095-8956(86)90030-4}{Graph minors.
  {V}. {E}xcluding a planar graph}.
\newblock \emph{J. Combin. Theory Ser. B}, 41(1):92--114, 1986{\natexlab{b}}.

\bibitem[{Robertson and Seymour(1995)}]{RS-XIII}
\textsc{Neil Robertson and Paul Seymour}.
\newblock \href{https://doi.org/10.1006/jctb.1995.1006}{Graph minors. {XIII}.
  {T}he disjoint paths problem}.
\newblock \emph{J. Combin. Theory Ser. B}, 63(1):65--110, 1995.

\bibitem[{Seymour(2016)}]{Seymour16}
\textsc{Paul Seymour}.
\newblock \href{https://doi.org/10.1016/j.jctb.2015.08.002}{Tree-chromatic
  number}.
\newblock \emph{J. Combin. Theory Series B}, 116:229--237, 2016.

\bibitem[{Shahrokhi(2013)}]{Shahrokhi13}
\textsc{Farhad Shahrokhi}.
\newblock \href{http://arxiv.org/abs/1502.06175}{New representation results for
  planar graphs}.
\newblock In \emph{29th European Workshop on Computational Geometry
  \textup{(EuroCG 2013)}}, pp. 177--180. 2013.
\newblock arXiv:1502.06175.

\bibitem[{Whitney(1933)}]{Whitney-AJM33c}
\textsc{Hassler Whitney}.
\newblock \href{http://www.jstor.org/stable/2371127}{2-{I}somorphic graphs}.
\newblock \emph{Amer. J. Math.}, 55(1-4):245--254, 1933.

\bibitem[{Yolov(2018)}]{Yolov18}
\textsc{Nikola Yolov}.
\newblock \href{https://doi.org/10.1137/1.9781611975031.16}{Minor-matching
  hypertree width}.
\newblock In \textsc{Artur Czumaj}, ed., \emph{Proc. 29th Annual {ACM-SIAM}
  Symposium on Discrete Algorithms \textup{(SODA)}}, pp. 219--233. {SIAM},
  2018.

\end{thebibliography}

\def\soft#1{\leavevmode\setbox0=\hbox{h}\dimen7=\ht0\advance \dimen7
  by-1ex\relax\if t#1\relax\rlap{\raise.6\dimen7
  \hbox{\kern.3ex\char'47}}#1\relax\else\if T#1\relax
  \rlap{\raise.5\dimen7\hbox{\kern1.3ex\char'47}}#1\relax \else\if
  d#1\relax\rlap{\raise.5\dimen7\hbox{\kern.9ex \char'47}}#1\relax\else\if
  D#1\relax\rlap{\raise.5\dimen7 \hbox{\kern1.4ex\char'47}}#1\relax\else\if
  l#1\relax \rlap{\raise.5\dimen7\hbox{\kern.4ex\char'47}}#1\relax \else\if
  L#1\relax\rlap{\raise.5\dimen7\hbox{\kern.7ex
  \char'47}}#1\relax\else\message{accent \string\soft \space #1 not
  defined!}#1\relax\fi\fi\fi\fi\fi\fi}

\end{document}